\documentclass[10pt,reqno]{amsart}
\usepackage{amssymb,amsmath,amsthm,amsbsy,amsfonts, titletoc,  mathrsfs, mathabx, verbatim, euscript, enumerate, mathtools, epsfig, color, hyperref, geometry,txfonts}

\newcommand{\R}{{\mathbb R}}

\newcommand{\Z}{{\mathbb Z}}
\newcommand{\ve}{\varepsilon}
 
\newcommand{\mB}{{\mathcal B}} 
\newcommand{\mC}{{\mathcal C}} 
\newcommand{\cB}{{\mathscr B}} 
\newcommand{\cC}{{\mathscr C}} 
\newcommand{\sD}{{\mathscr D}}
\newcommand{\ld}{\lambda}

\newcommand{\be}{\beta}

\newcommand{\D}{\nabla}
\newcommand{\La}{\Delta} 
\newcommand{\Div}{\operatorname{div}} 

\newcommand{\q}{\quad}
\newcommand{\qq}{\qquad}
\newcommand{\eq}{\begin{equation}}
\newcommand{\eeq}{\end{equation}}
\newcommand{\1}{\partial}

\newtheorem{thm}{Theorem}[section]
\newtheorem{lemma}[thm]{Lemma}
\newtheorem{cor}[thm]{Corollary}
\newtheorem{remark}[thm]{Remark}
\newtheorem{Claim}[thm]{Claim}
\newtheorem{prop}[thm]{Proposition}

\newtheorem*{notation}{Notation}
\errorcontextlines=0 \numberwithin{equation}{section}
\numberwithin{thm}{section}
\large
\begin{document}
\title[Vanishing  behavior]{Vanishing time behavior of   solutions to the fast diffusion equation}
 
\author[Kin Ming Hui]{Kin Ming Hui}
\address{Kin Ming Hui:
Institute of Mathematics, Academia Sinica, Taipei, Taiwan, R.O.C.}
\email{kmhui@gate.sinica.edu.tw}

\author{Soojung Kim}
\address{Soojung Kim:
School of Mathematics, Korea Institute for Advanced Study,  Seoul 02455,    Republic of Korea}
\email{soojung26@gmail.com; soojung26@kias.re.kr}

\date{\today}
\date{Nov 11, 2018}
 
\begin{abstract}
Let $n\geq 3$, $0< m<\frac{n-2}{n}$ and $T>0$. We construct positive solutions to the fast diffusion equation $u_t=\Delta u^m$ in $\mathbb{R}^n\times(0,T)$, which vanish at time $T$.  By introducing a scaling parameter $\beta$ inspired by \cite{DKS}, we study the second-order asymptotics of the self-similar solutions   associated with  $\beta$ at spatial infinity.  We also investigate the asymptotic behavior of  the solutions to the fast diffusion equation  near the vanishing time $T$, provided that  the initial value  of the solution is close to the initial value of some self-similar solution  and satisfies some proper  decay  condition at infinity. Depending on the range of the parameter $\beta$,  we prove that the rescaled solution   converges either to a  self-similar profile  or to  zero as $t\nearrow T$. The former implies asymptotic stabilization towards a self-similar solution, and the latter is a new vanishing phenomenon even for the case $n\ge3$ and $m=\frac{n-2}{n+2}\,$ which corresponds to the Yamabe flow on $\mathbb{R}^n$ with metric $g=u^{\frac{4}{n+2}}dx^2$. 
\end{abstract}

\maketitle

\tableofcontents

\section{Introduction}

The equation 
\begin{equation}\label{eq-fde}
\left\{\begin{aligned}
u_t&=\La u^m\qquad\mbox{ in }\R^n\times(0,T)\\
u(x,0)&=u_0(x)\quad\,\,\,\,\mbox{in }\R^n
\end{aligned}
\right.
\end{equation}  
appears in many physical models. When $m=1$, \eqref{eq-fde} is the classical heat equation.  When $m>1$, \eqref{eq-fde} is the porous medium equation which models the flow of gases or liquid through porous media. When $0<m<1$, \eqref{eq-fde} is the fast diffusion equation appearing in plasma physics.   The fast diffusion equation  also arises  in the study of  Yamabe flow in geometry.  
Let $g=u^{\frac{4}{n+2}}dx^2$  be   a conformally flat metric on $\R^n$ ($n\ge 3$),  which satisfies the Yamabe flow
\begin{equation}\label{eq-Yamabe}
\frac{\partial g}{\partial t}=-Rg \qquad\mbox{ for} \,\,0<t<T. 
\end{equation}
Here $R$ is the scalar curvature with respect to  the metric $g$.  Since 
$$
R=-\frac{4(n-1)}{n-2} \, u^{-1} \La u^{\frac{n-2}{n+2}}
$$    (see \cite{SY}), 
by \eqref{eq-Yamabe} $u$ satisfies 
\begin{equation*}
u_t=\frac{n-1}{m}\Delta u^m\quad\mbox{ in }\,\,\R^n\times (0,T)
\end{equation*}
with $m=\frac{n-2}{n+2}$;   \cite{DKS,dPS,Y}. After a rescaling, this  is equivalent to 
\begin{equation}\label{fast-diff-eqn}
u_t=\Delta u^m\quad\mbox{ in } \,\,\R^n\times (0,T).
\end{equation}

It is   well-known that there is a big difference in the behavior of the solutions to \eqref{eq-fde} for the cases $m>1$, $0<m<\frac{(n-2)_+}{n}$ and $\frac{(n-2)_+}{n}<m<1$. When $m>1$ and the initial value $u_0$ is non-negative and has compact support,   the solution $u$ of \eqref{eq-fde} will have compact support for all time of the existence of solution \cite{A}. On the other hand, as proved by M.A.~Herrero and M.~Pierre \cite{HP} when   $(n-2)_+/n< m<1$, \eqref{eq-fde}
has a unique global positive smooth solution in $\R^n\times(0,\infty)$  for any  initial value $0\lvertneqq u_0\in L^1_{loc}(\R^n)$. 

In the subcritical case when  $n\geq3$ and $ 0<m<\frac{n-2}{n}$,  the  Barenblatt solution (cf. \cite{DS})
\begin{equation}\label{eq-barenblatt}
\mB_{k}(x,t)=(T-t)^{\alpha}\left(\frac{C_* }{\,k^2+ \left|(T-t)^{\beta_1}x\,\right|^2}\,\right)^{\frac{1}{1-m}}\quad  \forall (x,t)\in \R^n \times (0,T)
\end{equation}
with  $k>0$ and $T>0$,  
is a  self-similar solution of \eqref{eq-fde} which vanishes identically at time $T$.  Here 
\begin{equation}\label{eq-def-beta1}
\beta_1=\beta_1(m)=\frac{1}{n-2-nm},\qquad \alpha=\frac{2\beta_1+1}{1-m},
\end{equation}
and 
\begin{equation}\label{c-star-defn}
C_*=C_*(m)=\frac{2m(n-2-nm)}{1-m}.  
\end{equation}  
Putting $k=0$ in \eqref{eq-barenblatt},  we obtain a singular solution   
\begin{equation}\label{eq-barenblatt-sing}
\mC(x,t)=\left(\frac{C_*(T-t)}{|x|^2}\right)^{\frac{1}{1-m}}\qquad \forall (x,t)\in (\R^n\setminus\{0\})\times (0,T)
\end{equation}
of  \eqref{fast-diff-eqn}.

When $n\ge 3$ and $0<m<\frac{n-2}{n}$, it is shown  that there exist initial values such that the corresponding positive smooth solutions of \eqref{eq-fde} vanish at time $T$  in  \cite{DS,DKS,GP}, while  there exist initial values such that \eqref{eq-fde} has unique global positive smooth solutions in $\R^n\times(0,\infty)$ in \cite{Hsu2}. In this   case when  $n\ge 3$ and $0<m<\frac{n-2}{n}$, the asymptotic extinction behaviors  of the vanishing solutions to  \eqref{eq-fde}  with initial values satisfying some   decay condition
 have been  studied   by V.A.~Galaktionov, L.A.~Peletier  \cite{GP}, M. del Pino, M.~S\'aez  \cite{dPS}, P.~Daskalopoulos, J.~King, N.~Sesum, \cite{DS,DKS},  K.M.~Hui \cite{Hui},  A.~Blanchet, M.~Bonforte,  J.~Dolbeault, G.~Grillo, J.L.~Vazquez \cite{BBDGV,BDGV}, M. Fila, J.L.~Vazquez, M.~Winkler, E.~Yanagida \cite{FKW,FVW, FVWY, FW}, etc.

Asymptotic stabilization to   Barenblatt solutions when   $n\geq3$ and $ 0<m<\frac{n-2}{n}$   was investigated    in   \cite{DS}, \cite{BBDGV} and \cite{FW}.     It was proved that  some rescaled function of the solution to \eqref{eq-fde} converges to  the  rescaled Barenblatt solution    as $t$ approaches the extinction time $T$, provided that the initial value of the solution  is close to the initial value of some  Barenblatt solution.


Let $n\geq 3$,  $0<m< \frac{n-2}{n}$  and $\alpha$, $\beta$ satisfy
\begin{equation}\label{eq-beta-exist}
 \beta\, \geq \, \beta_e(m):=\frac{m}{n-2-nm}, \qquad\alpha = \frac{2\beta+1}{1-m}. 
\end{equation}
By Theorem 1.1 of \cite{Hsu1},  for any $\lambda>0$, there exists a unique radially symmetric solution $f_{\lambda}$  to the elliptic  problem   
\begin{equation}\label{eq-ellip}
\left\{\begin{aligned}
&\La f^m+\alpha f+\beta  y \cdot\D f=0,\quad  f>0\qquad\hbox{in \,\,$\R^n$}\\
&f(0)=\lambda^{\frac{2}{1-m}}.
\end{aligned}\right.
\end{equation} 
Then
\begin{equation}\label{eq-U-ld}
U_\lambda(x,t):=(T-t)^\alpha f_{\lambda}\left((T-t)^\beta x\right)\qquad\forall (x,t)\in\R^n\times(0,T)
\end{equation}
 is a self-similar solution to \eqref{fast-diff-eqn}.
In particular, if $\beta=\beta_1(m)=\frac{1}{n-2-nm}$, then $\alpha=n\beta$ and for any $\lambda>0$,
\begin{equation}\label{eq-barenblatt-scaled}
f_{\lambda}(y)= \left(\frac{C_* }{k^2+|y|^2}\right)^{\frac{1}{1-m}}=:\cB_{k}(y)\qquad\hbox{ with }\,\,k=\sqrt{C_*}/\lambda,
\end{equation}  
and then  the corresponding self-similar solution $U_\lambda$  coincides with the Barenblatt solution $\mB_{k}$  in \eqref{eq-barenblatt} with  $k=\sqrt{C_*}/\lambda$.
Similarly, the singular Barenblatt solution $\mC(x,t)$ in \eqref{eq-barenblatt-sing} satisfies 
\begin{equation*}
\mC(x,t)=(T-t)^\alpha \cC\left((T-t)^\beta x\right)\qquad\forall (x,t)\in(\R^n\setminus\{0\})\times(0,T)
\end{equation*}
with
\begin{equation}\label{eq-def-tildeC}
\cC(y):= \left(\frac{C_*}{|y|^{2}}\right)^{\frac{1}{1-m}} \quad\forall y\in\R^n\setminus\{0\}
\end{equation}
for any  $\beta>0$, $\alpha = \frac{2\beta+1}{1-m}$ and $0<m<1$.

In \cite{DKS}, the authors studied    vanishing  solutions  to \eqref{eq-fde} including self-similar solutions in \eqref{eq-U-ld} for the case $n\ge 3$ and  $m=\frac{n-2}{n+2}$ which concerns    the Yamabe flow  of a conformally flat metric $g=u^{\frac{4}{n+2}}dx^2$ on $\R^n$. So the singularity formation of   conformally flat solutions to the   Yamabe flow at a singular  time can be described  by     the  extinction  profiles of  solutions to the fast diffusion equation   \eqref{eq-fde}.   

In this paper we will extend    the results of P. Daskalopoulos, K.M.~Hui, J. King and N. Sesum in  \cite{DS,DKS,Hui}, and investigate various  asymptotic  behavior of  solutions to  \eqref{eq-fde}  near the extinction time, provided that $n\geq 3$,    $0<m<\frac{n-2}{n}$, and the initial value  $u_0$ is   close to  the initial value  of some self-similar solution  $U_{\lambda}$  with 
\begin{equation*}
u_0(x)=\left(\frac{C_*}{|x|^2}\right)^{\frac{1}{1-m}}\big(1+o(1)\big)\qquad \hbox{as $|x|\to\infty.$}
\end{equation*}
Unless stated otherwise, we will now let $n\ge 3$, $0<m<\frac{n-2}{n}$ and $\alpha$, $\beta$ satisfy \eqref{eq-beta-exist} throughout the paper.
In   light of \eqref{eq-U-ld}, for any solution $u$ to \eqref{eq-fde} we define  the rescaled function $\tilde u$  of $u$ as
\begin{equation}\label{u-tilde-defn}
\tilde u(y,\tau):=(T-t)^{-\alpha}u\left((T-t)^{-\beta}y, T(1-e^{-\tau})\right)\qquad \forall (y,\tau)\in\R^n\times[0,\infty)
\end{equation}
with new variables
$$
y= (T-t)^{\beta}x\qquad \hbox{and}\qquad \tau=-\log \{(T-t)/ T\}.
$$
Then     $\tilde u$ satisfies 
\begin{equation}\label{eq-fde-rescaled}
\tilde u_\tau=\La \tilde u^m+\alpha\tilde u+\beta y\cdot\nabla \tilde u \qquad \hbox{in \,\,$\R^n\times (0,\infty)$.}\\
\end{equation} 
The asymptotic analysis of the solution   $u$  near the extinction time $T>0$ then   is  equivalent to understand  the long-time asymptotics  of the rescaled solution  $\tilde u$ as $\tau\to\infty$.   Note that the solution $f_\lambda$ of \eqref{eq-ellip} is a stationary solution to \eqref{eq-fde-rescaled} for any $\lambda>0$. 

In the first part of this paper, we will establish the second-order asymptotics of $f_{\lambda}(r)$ as $r=|y|\to\infty$ and study the comparison properties of $f_{\lambda}$.  By  Theorem 1.1 of \cite{Hsu3}, $f_{\lambda}$    satisfies
\begin{equation}\label{eq-1st-asymp-ellip}
\lim_{r\to\infty} r^{2}f_{\lambda}^{1-m}(r)=C_*>0
\end{equation}
where $r=|y|$ and $C_*$ is given by    \eqref{c-star-defn}. As  observed in  Subsection \ref{subsec-2nd-asymp},  in light of    \eqref{eq-1st-asymp-ellip}, the second-order  asymptotics of  $f_{\lambda}$ can be deduced from the study of the   linearization problem  of      \eqref{eq-ellip}  around the function $\cC(y)$  given by  \eqref{eq-def-tildeC}, 
which  leads to        the    characteristic equation 
\begin{equation}\label{eq-char}
\gamma^2- \frac{A_0(m,\beta)}{1-m}\, \gamma+\frac{2(n-2-nm)}{1-m}=0 .
\end{equation}
Here
\begin{equation}\label{eq-def-A_0(be)}
A_0(m,\beta)=n-2-(n+2)m+2\beta(n-2-nm).
\end{equation}
Let
\begin{equation*}
\beta_2=\beta_2(m)=\sqrt{\frac{2(1-m)}{n-2-nm}}+ \frac{(n+2)m-(n-2)}{2(n-2-nm)}
\end{equation*}
and
\begin{equation}\label{eq-def-beta0}
\begin{aligned}
\beta_0=\beta_0(m)&=\max\left(\beta_2(m),\beta_e(m)\right). 
\end{aligned}
\end{equation} 
Note that
\begin{equation}\label{A0>0}
A_0(m,\beta)\geq A_0(m,\beta_2(m))=\sqrt{8(n-2-nm)(1-m)}>0\qquad\forall \beta\geq \beta_0(m).
\end{equation}
Hence if  $\beta\geq \beta_0(m)$, the two positive roots $\gamma_2=\gamma_2(m,\beta)\geq \gamma_1=\gamma_1(m,\beta)>0$ of \eqref{eq-char} are   given by
 \begin{equation}\label{eq-roots-gamma-intro}
 \gamma_j(m,\beta)= \frac{1}{2(1-m)}\left\{\,A_0(m,\beta)+(-1)^j\sqrt{A_0(m,\beta)^2-8(n-2-nm)(1-m)}\right\}\,\,\,\forall j=1,2. 
\end{equation}  
Under the assumption that  $\beta\geq \beta_0(m) $, the positivity of  the  two   roots $\gamma_1$,  $\gamma_2$  implies a    non-oscillatory behavior of the second-order term in the  asymptotic expansion of $f_{\lambda}$ at $r=\infty$. Unless stated otherwise, we will restrict ourselves to the case $\beta>\beta_0(m)$ with the two positive roots $\gamma_2> \gamma_1 >0$ of \eqref{eq-char} throughout the paper.    We will also assume that  $\beta_1(m)$, 
 $\beta_e(m)$ and $\beta_0(m)$    are given by \eqref{eq-def-beta1}, \eqref{eq-beta-exist} and \eqref{eq-def-beta0}, respectively. 
 
 \medskip
 Now we are ready to state our first   theorem. 

\begin{thm}[Second-order asymptotics of self-similar profile  at infinity]\label{thm-self-sol-2nd-asymp}
Let  $n\ge3$, $ 0<m<\frac{n-2}{n}$, $\beta> \beta_0(m)$ and $\alpha$ be given by \eqref{eq-beta-exist}. For any $\lambda>0$, let $f_{\lambda}$ be the unique radially symmetric solution of \eqref{eq-ellip}. Then the following holds.

\begin{enumerate}[(a)]
\item    
If 
 \begin{equation} \label{eq-cond-mono}\tag{C1}
 \begin{aligned}
 \mbox{either}&\hbox{ (i) \,$n\geq 3$, \,$0< m<\frac{n-2}{n}$\, and\, $\beta>\beta_1(m)$ }\\
 \hbox{or}\quad  &\hbox{ (ii) \,$n> 4$, \, $0<m<\frac{n-4}{n-2}$\, and\,   $\beta_0(m)<\beta\leq\beta_1(m)$} 
 \end{aligned}
 \end{equation}  
holds, then there exists a constant $B>0$ such that   for any $\lambda>0$,
$$
f_{\lambda}(r)=\left(\frac{C_*}{ r^2}\right)^{ \frac{1}{1-m}}\Big\{\,1- B_\lambda\, r^{-\gamma}+o (r^{-\gamma})\,\Big\}\quad\mbox{ as \,} r\to\infty
$$
with $\gamma=\gamma_1(m,\beta)$ and  $B_\lambda=B \lambda^{-\gamma}$.
 
\item  
If 
\begin{equation} \label{eq-cond-non-mono}\tag{C2}
 \begin{aligned}
 \mbox{either}&\hbox{ (i) \,$3\leq n\leq 4$, \,$0<m<\frac{n-2}{n}$\, and\, $\beta_0(m)<\beta<\beta_1(m)$, }\\
 \hbox{or}\quad  &\hbox{ (ii) \,$n >4$, \, $\frac{n-4}{n-2} < m<\frac{n-2}{n}$\, and \,  $\beta_0(m)<\beta<\beta_1(m)$} 
 \end{aligned}
 \end{equation} 
holds, then there exists a constant $B>0$ such that   for any $\lambda>0$,
\begin{equation*}
f_{\lambda}(r)= \left(\frac{C_*}{ r^2}\right)^{\frac{1}{1-m}}\Big\{\,1+ B_\lambda\, r^{-\gamma}+o (r^{-\gamma})\,\Big\}\quad\mbox{ as \,} r\to\infty
\end{equation*} 
with $\gamma=\gamma_1(m,\beta)$ and  $B_\lambda=B \lambda^{-\gamma}$.

\item 
If 
\begin{equation} \label{eq-cond-mono-beta1}\tag{C3}
 \begin{aligned}
\mbox{either} &\hbox{ (i) \,$3\leq n\leq 4$, \,$0<  m<\frac{n-2}{n}$\, and\, $\beta=\beta_1(m)$, }\\
 \hbox{or}\quad  &\hbox{ (ii) \,$n> 4$, \, $ \frac{n-4}{n-2} <m<\frac{n-2}{n}$\, and \,  $\beta=\beta_1(m)$} 
 \end{aligned}
 \end{equation}
holds,  then there exists a constant $B>0$ such that   for any $\lambda>0$,
$$
f_{\lambda}(r)=\left(\frac{C_*}{ r^2}\right)^{ \frac{1}{1-m}}\Big\{\,1- B_\lambda\, r^{-\gamma}+o (r^{-\gamma})\,\Big\}\quad\mbox{ as } r\to\infty
$$
with  $\gamma =\gamma_2(m,\beta_1(m))=2$ and $B_\lambda=B \lambda^{-2}$. 
 In this case, the self-similar profile $f_{\lambda} $ satisfies  \eqref{eq-barenblatt-scaled}. 
\end{enumerate}
\end{thm}

\begin{remark}
\hspace{1cm}

\begin{enumerate}[(i)]

\item 
\cite[Theorem 1.1]{DKS} can be regarded as a special case of  Theorem \ref{thm-self-sol-2nd-asymp} with  $m=\frac{n-2}{n+2}$ and $n\ge 3$. We also mention that  the asymptotic  result  of   \cite[Theorem 1.2]{DKS}    for    self-similar profiles with a singularity at the origin    in the case   $m=\frac{n-2}{n+2}$  was extended by K.M.~Hui in \cite{Hui} to  the subcritical range $0<m<\frac{n-2}{n}$, $n\geq3$.
  
\item For $n>4$, the critical exponent $m=\frac{n-4}{n-2}$   appears  in   Theorem \ref{thm-self-sol-2nd-asymp}. In fact, Lemma \ref{lem-beta0-beta1} shows that the parameters $ \beta_0(m)$ and $ \beta_1(m)$ are equal to each other if and only if $m=\frac{n-4}{n-2}$ and $n>4$.  As observed by P.~Daskalopoulos and N.~Sesum \cite{DS} when $\beta=\beta_1(m)$ and $n\ge 3$, the condition  $\frac{n-4}{n-2}<m<\frac{n-2}{n}$  implies  that the    difference of two  Barenblatt profiles $\cB_k$ is integrable. This result will be generalized in Lemma \ref{lem-diff-solitons-integrable} using the results obtained in Theorem \ref{thm-self-sol-2nd-asymp}. 

\end{enumerate}
\end{remark}

In the following theorem, we examine the  monotonicity  of $f_{\lambda}$ with respect to $\lambda>0$ and the integrability  of $f_{\lambda_1}-f_{\lambda_2}$ for any $\lambda_2>\lambda_1>0$.    
  
\begin{thm}\label{thm-self-sol-2nd-asymp-mono}
Let  $n\ge3$, $ 0<m<\frac{n-2}{n}$, $\beta> \beta_0(m)$ and $\alpha$ be given by \eqref{eq-beta-exist}. For any $\lambda>0$, let $f_{\lambda}$ be the unique radially symmetric solution of \eqref{eq-ellip}. Then the following holds.
\begin{enumerate}[(a)]

\item  (Monotone increasing case) If \eqref{eq-cond-mono} holds,  then
\begin{equation}\label{eq-(A)-mono}
f_{\lambda_1} (r)< f_{\lambda_2}(r)\qquad \forall r\geq0,\,\,\lambda_2>\lambda_1>0.
\end{equation}
Moreover    
$$
f_{\lambda_1}-f_{\lambda_2}\not \in L^1( \R^n) \qquad\forall \lambda_2>\lambda_1>0.
$$

\item (Non-monotone increasing case) If \eqref{eq-cond-non-mono} holds, there exists a constant $R_0>0$ such that  
\begin{equation}\label{eq-(B)-non-mono}
f_{\lambda_1} (r)> f_{\lambda_2}(r)\qquad \forall r\geq R_0/\lambda_1,\,\,\lambda_2>\lambda_1>0.
\end{equation}
Moreover 
\begin{equation}\label{eq-(B)-integ}
f_{\lambda_1}-f_{\lambda_2}\in L^1( \R^n) \qquad\forall \lambda_2>\lambda_1>0.
\end{equation}

\item  (Monotone increasing and integrable case) If \eqref{eq-cond-mono-beta1} holds,  then both  \eqref{eq-(A)-mono} and  \eqref{eq-(B)-integ} hold.
\end{enumerate}
\end{thm}
 
Now we establish the existence  of solutions  to  \eqref{eq-fde} which  vanish  at time $T$ when   the initial value $u_0$  is close to the initial value of a  self-similar solution  $U_\lambda$ given by \eqref{eq-U-ld}.  
Moreover by using the properties of $f_{\lambda}$ in Theorem \ref{thm-self-sol-2nd-asymp} and Theorem \ref{thm-self-sol-2nd-asymp-mono}, we address the vanishing asymptotics   of  such solutions near the extinction time. We will let   $T>0$ and $U_\lambda$ be given by \eqref{eq-U-ld} for the rest of the paper.

 \begin{thm}[Existence]\label{thm-existence-fde}
 Let  $n\ge3$, $ 0<m<\frac{n-2}{n}$, $\beta> \beta_e(m)$ and $\alpha$ be given by \eqref{eq-beta-exist}. Suppose  $u_0$   satisfies 
\begin{equation}\label{u0<U-lambda_2}
0\leq u_0\leq U_{\lambda_2}(\cdot,0)  \qquad\hbox{ in }\,\,\R^n  
\end{equation}
and 
\begin{equation}\label{eq-initial-soliton-L1}
u_0-U_{\lambda_0}(\cdot,0)\in L^1(\R^n)
\end{equation} 
for some constants  $\lambda_2>0$, $\lambda_0>0$.
Then there exists a unique   solution $u$ to the Cauchy problem \eqref{eq-fde}  which  satisfies
\begin{equation}\label{ab-ineqn}
u_t\leq \frac{u}{(1-m)t}\qquad\hbox{ in  }\,\,\R^n\times(0,T). 
\end{equation}
Moreover the following  holds. 
\begin{enumerate}[(i)]
\item  
\begin{equation}\label{u-u-lambda-ineqn2}
0 < u(x,t)\le U_{\lambda_2}(x,t)\qquad\hbox{ in }\,\,\R^n\times(0,T).
\end{equation}

\item For any $0<t<T$, 
\begin{equation}\label{eq-exist-sol-contraction}
\int_{\R^n} |u(x,t)-U_{\lambda_0}(x,t)|dx\leq \int_{\R^n} |u_0(x)-U_{\lambda_0}(x,0)|dx. 
\end{equation} 

\item
In addition, if $u_0$ also satisfies 
\begin{equation}\label{u0>U-lambda}
u_0\ge U_{\lambda_1}(\cdot,0) \qquad\hbox{ in }\,\,\R^n  
\end{equation}
for some constant $\lambda_1>0$, then
\begin{equation}\label{u>u-lambda-ineqn1}
u(x,t)\ge U_{\lambda_1}(x,t)\qquad\hbox{ in }\,\,\R^n\times(0,T).
\end{equation}
\end{enumerate}
\end{thm}
  
In  Theorem \ref{thm-existence-fde},  we assume that $u_0$ is a $L^1$-perturbation of  $U_{\lambda_0}(\cdot,0)$ and   is bounded from above by $U_{\lambda_2}(\cdot,0)$. It is worth noting that the upper bound of  $U_{\lambda_2}(\cdot,0)$ is responsible  for the solution $u$ of \eqref{eq-fde} to have the same  extinction time $T$ as the self-similar solution $U_{\lambda_2}$ (cf. Theorem 1.6 of \cite{DKS}).   

 \medskip
 
Next, we will study  the extinction profiles of  the solutions to  \eqref{eq-fde}  given in   Theorem \ref{thm-existence-fde} under  the condition   \eqref{eq-cond-mono},  \eqref{eq-cond-non-mono} or    \eqref{eq-cond-mono-beta1} using results of Theorem  \ref{thm-self-sol-2nd-asymp} and Theorem \ref{thm-self-sol-2nd-asymp-mono}. 
  As mentioned before,  asymptotic stabilization to the Barenblatt profiles  $\cB_k$  in the rescaled variables (for the case   $\beta=\beta_1(m)$)  was established  in \cite{DS,BBDGV}, which corresponds to the case \eqref{eq-cond-mono-beta1}.  Henceforth in this paper we restrict our attention to cases     \eqref{eq-cond-mono} and  \eqref{eq-cond-non-mono}. 
 
 In the monotone increasing case \eqref{eq-cond-mono}, we obtain asymptotic stability of the solutions to  \eqref{eq-fde} in rescaled variables as $t$ approaches  the extinction time $T$.  That is,   the rescaled  solution $\tilde u $    converges   to  a self-similar profile $f_{\lambda}$ locally uniformly in $\R^n$ as the rescaled time $\tau\to\infty$. 

\begin{thm}[Convergence to a self-similar profile]\label{thm-uniform-convergence-tilde-u-compact}  
Suppose   \eqref{eq-cond-mono} holds and  $u_0$ satisfies  \eqref{eq-initial-soliton-L1} and 
\begin{equation} \label{u-lower-upper-bd-ini}
U_{\lambda_1}(\cdot,0)\leq  u_0\leq  U_{\lambda_2}(\cdot,0)\qquad\hbox{ in } \,\,\R^n
\end{equation}
for  some  constants   $\lambda_2\geq \lambda_0\geq \lambda_1>0$.    Let    $u $   be    the unique   solution of  \eqref{eq-fde}  given by  Theorem \ref{thm-existence-fde}. 
Then 
the rescaled solution  $\tilde u  (\cdot, \tau)$   given by \eqref{u-tilde-defn}    converges to $f_{\lambda_0}$ uniformly in $C^2(K)$ as $\tau\to\infty$ for any  compact subset $K\subset\R^n$.  
Moreover, if (i) of \eqref{eq-cond-mono} holds, then   
\begin{equation}\label{eq-resc-sol-exp}
\left\|\tilde u(\cdot,\tau)-f_{\lambda_0} \right\|_{L^1(\R^n)}\leq   T^{n\beta-\alpha} e^{-(n\beta-\alpha)\tau}\left\|u_0 -U_{\lambda_0}(\cdot,0)\right\|_{L^1(\R^n)}\qquad\forall\tau>0
\end{equation}  
with $n\beta-\alpha=\frac{1}{1-m}\left(\frac{\beta}{\beta_1(m)}-1\right)  >0$. 
\end{thm}

\begin{remark}
\hspace{1cm}

\begin{enumerate}[(i)] 
\item The weighted $L^1$-contraction principle for the rescaled solution  $\tilde u$ in  Proposition \ref{prop-strong-contraction}  is a key ingredient in the study of convergence to a self-similar solution   in  Theorem \ref{thm-uniform-convergence-tilde-u-compact}.  Here the weighted $L^1$-space, $L^1(\cC^{p_0}; \R^n)$,  involves  the weight function  $\cC ^{p_0}(y)$,    
where $\cC(y)$ is given by \eqref{eq-def-tildeC} and  $p_0$ is  given by  \eqref{p0-defn}. Moreover  the $L^1$-integrability in \eqref{eq-initial-soliton-L1} implies that  $u_0- U_{\lambda_0} (\cdot,0) \in L^1(\cC^{p_0}; \R^n)$ since $\cC^{p_0}$ is integrable near the origin.

\item \cite[Theorem 1.2]{DS} can be regarded as a special case of Theorem \ref{thm-uniform-convergence-tilde-u-compact}  with $n>4$, $0<m < \frac{n-4}{n-2}$ and  $\beta=\beta_1(m)$. In \cite{DS}, the self-similar profiles $f_{\lambda}$ are explicitly given  by the rescaled Barenblatt profiles  $\cB_k$ in  \eqref{eq-barenblatt-scaled} as mentioned before.      Theorem 1.1 of \cite{DS} dealt with  the convergence of $\tilde u$   to a rescaled  profile  $\cB_k$  in the case of  \eqref{eq-cond-mono-beta1}, where  the integrability of the difference of two rescaled   profiles (cf. (c) of Theorem \ref{thm-self-sol-2nd-asymp-mono}) plays a critical role in  the proof. 

\item  
 In  the proof   of  \cite[Theorem 1.2]{DS}  when $n>4$, $0<m < \frac{n-4}{n-2}$ and  $\beta=\beta_1(m)$,   the authors used  the  Osher--Ralston approach \cite{OR} in order to establish    the convergence in a suitable weighted $L^1$-space.      But, it seems difficult to apply  the  Osher--Ralston approach in our case      \eqref{eq-cond-mono}   with   $\beta<\beta_1(m)$. In fact,  when  $\beta=\beta_1(m)$    (i.e., $\alpha=n\beta$)  in    \cite[Theorem 1.2]{DS},  the $L^1$-contraction principle  for   the solutions $u$ of  \eqref{eq-fde}   yields   the $L^1$-contraction principle  for     the rescaled solutions  $\tilde u$ since 
 $$
  \left\|\tilde u(\cdot,\tau)-f_{\lambda_0}\right\|_{L^1(\R^n)} = \left\|u(\cdot,t)-U_{\lambda_0}(\cdot,t)\right\|_{L^1(\R^n)}
$$ 
with $\tau=-\log \{(T-t)/ T\}>0$.  The    $L^1$-contraction   principle  for   $\tilde u$   was  crucial to proceed with the argument of Osher--Ralston in  \cite[Theorem 1.2]{DS}.

\end{enumerate}

\end{remark}

In the non-monotone increasing case    \eqref{eq-cond-non-mono}, we obtain a  new asymptotic result  of    the  rescaled  solution $\tilde u $ as $\tau\to\infty$.  More precisely we obtain the following result.

\begin{thm}[Convergence to zero]\label{thm-uniform-convergence-tilde-u-0} 
Suppose  \eqref{eq-cond-non-mono} holds and $u_0$ satisfies   \eqref{eq-initial-soliton-L1} and
 \begin{equation} \label{eq-tilde-sol-upper-bd-ini}
 0\leq   \, u_0\,  \leq\min \big\{U_{\lambda_1}(\cdot,0),U_{\lambda_2}(\cdot, 0)\big\}  \qquad\hbox{ in } \,\,\R^n
\end{equation}
for some constants $\lambda_2>\lambda_1>0$, $\lambda_0>0$. Let $u$ be the unique solution of \eqref{eq-fde} given by  Theorem \ref{thm-existence-fde}. 
Then      the rescaled solution  $\tilde u (\cdot,\tau) $   given by \eqref{u-tilde-defn}   converges to zero uniformly on any  compact subset of $\R^n$ as $\tau\to\infty$.  
\end{thm}

\begin{remark}
\hspace{1cm}

\begin{enumerate}[(i)]

\item In the case of  \eqref{eq-cond-non-mono}, we conclude from Theorem \ref{thm-uniform-convergence-tilde-u-0} that asymptotic stabilization of the rescaled solution to   the self-similar profile $f_{\lambda }$  does not occur, which     differs from Theorem \ref{thm-uniform-convergence-tilde-u-compact}.    In fact, if we fix $\lambda>0$ and consider the  rescaled solution $\tilde u_\ve$  with the initial value   $\tilde u_{0,\ve}=\min \left\{f_{\lambda} ,f_{\lambda-\ve} \right\}  $   for   $0<  \ve< \lambda$, then by \eqref{eq-(B)-non-mono}  there exists some constant  $R>0$ such that     
 $\tilde u_{0,\ve}\equiv f_{\lambda } $   on $\R^n\setminus B_R$ for     small $ 0< \ve <\lambda/2$.  Then such $\tilde u_{0,\ve}$ is a small perturbation of $f_{\lambda } $, and  by Theorem \ref{thm-uniform-convergence-tilde-u-0},   the rescaled solution $\tilde u_\ve$    converges to zero as $\tau\to\infty$. This  implies that  
  $f_{\lambda}$ is not asymptotically stable with respect  to small perturbations.  
  
   \item  Theorem \ref{thm-uniform-convergence-tilde-u-0} is a new phenomenon of   the vanishing    behavior of the rescaled solutions    
   even in the particular case  $3\leq n< 6$, $m=\frac{n-2}{n+2}\,$  and $\beta_0(m)<\beta<\beta_1(m)$, which   is related  the conformally flat Yamabe flow; see \cite{DKS}.

\item  In light of the integrability result \eqref{eq-(B)-integ}  for  the case \eqref{eq-cond-non-mono}, 
   the condition \eqref{eq-initial-soliton-L1}   is equivalent to   $u_0-U_{\lambda}\in L^1(\R^n)$ for any $\lambda>0$.  

\end{enumerate}
\end{remark}

The rest of the paper is organized as follows. Section \ref{sec-profiles} is devoted to  the proofs of  Theorem  \ref{thm-self-sol-2nd-asymp} and Theorem \ref{thm-self-sol-2nd-asymp-mono}.    In Section \ref{sec-exist}, we establish  the existence of solution to  \eqref{eq-fde} in  Theorem \ref{thm-existence-fde}. In  Section \ref{sec-extinction}, we prove Theorem \ref{thm-uniform-convergence-tilde-u-compact} and  Theorem \ref{thm-uniform-convergence-tilde-u-0}  regarding  asymptotic results near the extinction time. 
 
\begin{notation}
Let us summarize the definitions and notations that are used in the paper.
\begin{itemize}
\item For any $0 \leq u_0\in L^1_{loc}(\R^n)$, we say that $u$ is a solution to the Cauchy problem \eqref{eq-fde} in $\R^n\times(0,T)$ if $u > 0$ in $\R^n\times(0,T) $  and $u$ is a classical solution of \eqref{fast-diff-eqn}      satisfying
 $$\lim_{\,t\to 0}\, \left\| u(\cdot, t)-u_0\right\|_{L^1(K)} =0 $$ for any compact set $ K\subset \R^n$.  

\item For   $x_0 \in \R^n$ and $R>0, $  we let 
 $B_R(x_0)=\left\{x\in\R^n \,: \,|x-x_0|<R\right\}$ and $B_R = B_R(0). $
 
 \item When there is no ambiguity, we will drop the dependence on $m$ and write $\beta_0$, $\beta_1$  for  $\beta_0(m)$, $\beta_1(m)$, etc.
\end{itemize}
\end{notation}

\section{Second-order asymptotics of self-similar profiles}\label{sec-profiles}


In this section, we will study the self-similar profile $f_{\lambda}$ which is the radially symmetric solution of \eqref{eq-ellip} for any $\lambda>0$. In particular we will prove  the second-order asymptotics of the self-similar profile $f_{\lambda}(r)$ as $r=|y|\to\infty$.

 \subsection{Existence   of self-similar profiles}


Firstly we recall  some results  of  \cite{Hsu1,Hsu3} on the existence, uniqueness  and the first-order asymptotics of the self-similar profiles which are   radially symmetric  solutions to the nonlinear elliptic problem \eqref{eq-ellip}. 

\begin{thm}[\cite{Hsu1,Hsu3}]\label{thm-cigar-soliton-existence}
Let  $n\geq 3$, $0<m< \frac{n-2}{n}$ and let $\alpha$, $\beta$ satisfy \eqref{eq-beta-exist}. For any  $\lambda>0$,  
  there exists a unique solution  $f_{\lambda}$ to the second-order ordinary differential equation   
\begin{equation}\label{eq-ode-ellip-initial}
\left\{\begin{aligned}
&(f^m)''+\frac{n-1}{r}(f^m)'+\alpha f+\beta  rf'=0,\quad f>0\qquad\hbox{in $ \,\,(0,\infty)$},\\
&f(0)=\lambda^{\frac{2}{1-m}},\qquad f'(0)=0.
\end{aligned}\right.
\end{equation}  
Moreover $f_{\lambda}$ satisfies \eqref{eq-1st-asymp-ellip} and
$$
\begin{aligned}
f'(r)<0\qquad&\forall r>0;\\
\alpha f(r)+\beta rf'(r)>0\qquad&\forall r\ge 0;\\
\left(r^{n-1}(f^m)'\right)'<0\qquad&\forall r\ge 0.
\end{aligned}
$$
\end{thm}

\begin{lemma}[Scaling property] \label{lem-scaling} 
Let  $n\geq 3$, $0<m< \frac{n-2}{n}$, $\lambda>0$ and let $\alpha$, $\beta$ satisfy \eqref{eq-beta-exist}. 
Let $f_{\lambda}$ be the unique solution to \eqref{eq-ode-ellip-initial}.   Then    
\begin{equation}\label{f-lambda-f1-eqn}
f_{\lambda}(r)=\lambda^{\frac{2}{1-m}}f_1(\lambda r)\qquad \forall r\geq0,
\end{equation}
where $f_1$ is  the  unique solution to  \eqref{eq-ode-ellip-initial} with $\lambda=1$. 
\end{lemma}
\begin{proof}
Let
$$
\tilde f(r)=\lambda^{\frac{2}{1-m}}f_1(\lambda r)\qquad \forall r\geq0.
$$
Then $\tilde f$ satisfies  \eqref{eq-ode-ellip-initial}. Since $f_{\lambda}$ is also a solution of \eqref{eq-ode-ellip-initial}, by uniqueness of solutions to  \eqref{eq-ode-ellip-initial} in Theorem \ref{thm-cigar-soliton-existence} the lemma follows.
\end{proof}

\subsection{Second-order asymptotics of self-similar profiles} \label{subsec-2nd-asymp} 
%

In this subsection we will use a modification of the proof of  \cite[Theorem 1.1]{DKS} and Section 3 of  \cite{Hui} to study the second-order asymptotics of the  self-similar profile $f_{\lambda}$.  
 
In light of  the first-order asymptotics \eqref{eq-1st-asymp-ellip}, we define a normalized  function $g_\lambda$ by  
$$
g_\lambda(s):= \left[{C_*}^{-\frac{1}{1-m}}r^{\frac{2}{1-m}}f_\lambda(r)\right]^m\qquad \forall s =\log r,\quad r>0.
$$
Then by \eqref{eq-ode-ellip-initial}, \eqref{eq-1st-asymp-ellip}, and the computation in Section 3 of \cite{Hsu1} and Section 3 of \cite{Hui}, $g_\lambda$  satisfies
\begin{equation}\label{eq-2nd-asymp-q}
\left\{
\begin{aligned}
&g''+\left(\frac{n-2-(n+2)m}{1-m}+\frac{\beta C_*}{m} g^{\frac{1}{m}-1}\right)g'+\frac{2m(n-2-nm)}{(1-m)^2}\left(g^{\frac{1}{m}}-g\right)=0\quad\mbox{in\,\, $\R$;}\\
&  \lim_{s\to-\infty}g (s)=0\quad\hbox{and}\quad \lim_{s\to\infty}g (s)=1.
\end{aligned}\right. 
\end{equation} 
We now  linearize \eqref{eq-2nd-asymp-q} around the constant $1$ by setting
$g_\lambda=1+w_\lambda$. 
Then $w_\lambda $ satisfies ((3.3) of \cite{Hui})
\begin{align}\label{eq-2nd-asymp-w-lin}
\left\{\begin{aligned}
& w'' +\left(\frac{n-2-(n+2)m}{1-m}+\frac{\beta C_*}{m} (1+w)^{\frac{1}{m}-1}\right)w' \\
& \qquad+\frac{2m(n-2-nm)}{(1-m)^2}\left(\,(1+w)^{\frac{1}{m}}-1-w\,\right)=0\qquad  \hbox{ in \,\, $\R$; } \\
& \lim_{s\to-\infty}w(s)=-1,\qquad  \lim_{s\to\infty} w(s)=0.
\end{aligned}\right. 
\end{align}  
Here we note that $w_\lambda>-1$ in $\R$ since $g_\lambda>0$ in $\R$.

We next define    the linearized operator $L$  for  the above equation \eqref{eq-2nd-asymp-w-lin}   around the zero  solution $w\equiv0$, as
\begin{equation*}
L[v]:=v''+ \frac{A_0(m,\beta)}{1-m} \,v'+\frac{2(n-2-nm)}{1-m}\,v
\end{equation*}
where  $A_0(m,\beta)$ is given by \eqref{eq-def-A_0(be)}.
Let 
\begin{equation}\label{phi-defn}
\phi(z):= (1+z)^{\frac{1}{m}}-1-\frac{z}{m}\qquad\forall z>-1,
\end{equation}
$w=w_\lambda$, and   
\begin{equation*}
\Phi(s):=\phi (w(s))\qquad\hbox{ for\,\, $s\in\R$.}
\end{equation*}
Then \eqref{eq-2nd-asymp-w-lin} can be written as
\begin{equation}\label{eq-w-linearized}
L[w]=h\quad \hbox{ in \,$\R$}
\end{equation}
where  
$$
h(s):=-C_*\left(\,{\beta}\, \Phi'(s)+\frac{1}{1-m}\,\Phi(s)\right).
$$   
The characteristic equation associated with the linearized operator $L$ is given by \eqref{eq-char} with characteristic
roots $ \gamma_1(m,\beta)$, $ \gamma_2(m,\beta)$ given by \eqref{eq-roots-gamma-intro}. We now assume that  $\beta>\beta_0(m)$. 
Then  by \eqref{A0>0}, 
\begin{equation}\label{eq-A_0-pos}
A_0(m,\beta)> \sqrt{8(n-2-nm)(1-m)}>0\qq \forall \beta>\beta_0(m),
\end{equation}
 and hence  the two positive roots of \eqref{eq-roots-gamma-intro}   satisfy $\gamma_2(m,\beta)>\gamma_1 (m,\beta)>0$.  

By the computation in \cite[Section 3]{Hui}, it holds that 
\begin{align}\label{eq-w-sol-lin-op}
w(s)&=\frac{1}{\gamma_2(m,\beta)-\gamma_1(m,\beta)}\left(\,e^{-\gamma_1s}\int_{-\infty}^se^{\gamma_1t}h(t)\,dt-e^{-\gamma_2s}\int_{-\infty}^se^{\gamma_2t}h(t)\,dt \,\right)\notag
\\
&=-C_0(m,\beta) \int_{-\infty}^s \left(\, A_1(m,\beta) e^{-\gamma_1 (s-t)}  - A_2(m,\beta) e^{-\gamma_2 (s-t)}  \,\right) \Phi(t)\,dt\qquad\forall s\in\R,
\end{align}
where
$$
\begin{aligned}
C_0(m,\beta)&:=  \frac{C_*}{  \gamma_2(m,\beta)-\gamma_1(m,\beta) }>0
\end{aligned} 
$$
and 
\begin{equation}\label{eq-def-A_j}
A_j(m,\beta):=\frac{1}{1-m}-\beta\,\gamma_j(m,\beta)\qquad\hbox{for \,\,$ j=1,2$}.\\
\end{equation}
Here  for simplicity, we denote   $ \gamma_j= \gamma_j(m,\beta)$ for $ j=1,2$.
  Since $\gamma_2(m,\beta)>\gamma_1(m,\beta)>0$,  we have 
\begin{equation}\label{eq-a1>a2} 
A_1(m,\beta)>A_2(m,\beta)\qquad\forall\beta>\beta_0(m). 
\end{equation} 

Observe that the function $\phi$ given by \eqref{phi-defn} is a strictly convex non-negative function  on $(-1,\infty)$ and $\phi (0)=\phi '(0)=0$. Hence there exist  constants $0<\delta_1<1$ and $c_2>c_1>0$   such  that 
\begin{equation}\label{phi-s2-eqn}
c_1 t^2\leq \phi(t)\leq c_2 t^2\qquad\forall |t|\le\delta_1.
\end{equation}
Then $\Phi(t)\geq 0$ for all $t\in\R$ since $w_\lambda>-1$ in $\R$, and  $\Phi(t)=0 $ if and only if $w(t)=0$. 
Let
\begin{equation}\label{delta2-defn}
\delta_2=\min \left(\delta_1,\frac{\gamma_2-\gamma_1}{2c_2C_*\beta}\right).
\end{equation}
Since    $\displaystyle\lim_{s\to\infty}w(s)=0$, there exists  a constant $s_0>0$  such  that 
\begin{equation}\label{w-small-at-infty}
|w(s)|\le\delta_2\qquad\forall s\ge s_0.
\end{equation}
By \eqref{phi-s2-eqn} and \eqref{w-small-at-infty}, we have 
\begin{equation}\label{eq-Phi-w^2}
c_1 w^2(s)\leq \Phi(s)\leq c_2 w^2(s)\qquad\forall s\ge s_0.
\end{equation}

In order to study the asymptotic behavior of $w$ given by \eqref{eq-w-sol-lin-op}, we  first  investigate the relation between  $\beta_0(m)$ and  $\beta_1(m)$, and the sign of  $A_1(m,\beta)$ and $A_2(m,\beta)$. 
    
\begin{lemma}\label{lem-beta0-beta1}
Let $n\geq3$ and  $0<m<\frac{n-2}{n} $. Then $\beta_1(m)\geq\beta_0(m)$ 
where the  equality  $\beta_1(m)=\beta_0(m)$ holds   if and only if $m=\frac{n-4}{n-2} $ and  $n>4$. 
\end{lemma}
\begin{proof}
By direct computation,  we have
\begin{align*}
2(n-2-nm)(\beta_1(m)-\beta_2(m))
&= {(n -nm-2m)- \sqrt{8(1-m)(n-2-nm)}}.
\end{align*}
Since
\begin{equation}\label{n-m-eqn}
\left(n -nm-2m \right)^2 - 8(1-m)(n-2-nm)=\left[n-4-(n-2)m\right]^2\geq0,
\end{equation}
it follows that $\beta_1(m)\geq \beta_2(m)$, and the equality holds if and only if  $m=\frac{n-4}{n-2}$ and $n>4$. Since $\beta_1(m)>\beta_e(m)$ and $\beta_0(m)=\max(\beta_2(m),\beta_e(m))$, the lemma follows.  
\end{proof}
    
\begin{lemma} \label{lem-A1-0}
Let  $n\geq3$ and  $ 0< m <\frac{n-2}{n} $.     Then the following  holds. 

\begin{enumerate}[(a)]
\item If 
\begin{equation}\label{A1-eqn1}
A_1(m,\beta)=0\quad\mbox{ for some }\beta\ge\beta_0(m),
\end{equation}
 then  $ \beta=\beta_1(m) $ and $m\ge\frac{n-4}{n-2}$.
\item $A_1\left(m,\beta_1(m)\right)>0$ and  $A_2\left(m,\beta_1(m)\right)=0$     if $0< m< \frac{n-4}{n-2}$ and $n>4$.
\item $A_1\left(m,\beta_1(m)\right)=0$ and $A_2\left(m,\beta_1(m)\right)=0$  if  $ m= \frac{n-4}{n-2} $ and $n>4$.
 
\item $A_1\left(m,\beta_1(m)\right)=0$ and $A_2\left(m,\beta_1(m)\right)<0$  if  $ \max\left(\frac{n-4}{n-2}, 0\right)< m<\frac{n-2}{n}$.
\end{enumerate}

\end{lemma}
\begin{proof}
Firstly, we show the statement (a). Note that by \eqref{eq-roots-gamma-intro} and \eqref{eq-def-A_j}, \eqref{A1-eqn1} is equivalent to 
\begin{equation}\label{A0-eqn1}
A_0(m,\beta)-\frac{2}{\beta}=\sqrt{A_0(m,\beta)^2-8(n-2-nm)(1-m)}\ge 0.
\end{equation}
Since 
\begin{align*}
&\left(A_0(m,\beta)-\frac{2}{\beta}\right)^2-A_0(m,\beta)^2+8(n-2-nm)(1-m)\\
=&4 \left(\frac{1}{\beta}-(n-2-nm)\right)\left(\frac{1}{\beta}+2m\right),
\end{align*}
\eqref{A0-eqn1} implies that   $\beta=\beta_1(m)$.  In light of \eqref{A0-eqn1},  we have
\begin{equation*}
A_0(m,\beta_1(m))-\frac{2}{\beta_1(m)}
=(n-2)m-(n-4)\ge 0,
\end{equation*}
and hence   $m\ge\frac{n-4}{n-2}$. Thus   (a) follows.

Since $A_0(m,\beta_1(m))=n-nm-2m$, by \eqref{eq-roots-gamma-intro}, \eqref{eq-def-A_j}, \eqref{n-m-eqn}, and a direct computation,  we have 
\begin{align*}
 \frac{2(1-m)A_j(m,\beta_1(m))}{\beta_1(m)}
&=  2(n-2-nm) - (n-nm-2m)\notag\\
&\quad   -(-1)^j\sqrt{(n-nm-2m)^2-8(n-2-nm)(1-m)} \\
&= n-4-(n-2)m-(-1)^j\left|n-4-(n-2)m\right|\qquad\forall j=1,2.
\end{align*}
Hence 
$$
\frac{2(1-m)A_1(m,\beta_1(m))}{\beta_1(m)}=\left\{
\begin{aligned}
&2\big\{(n-4)-(n-2)m\big\}>0\,\quad \hbox{if }\,\,0<m< \frac{n-4}{n-2}\,\mbox{ and }n>4,\\
&0\qquad\qquad\qquad\qquad\qquad\qquad \hbox{if }\,\,\frac{n-4}{n-2} \le m<\frac{n-2}{n},
\end{aligned}\right.
$$
and
$$
\frac{2(1-m)A_2(m,\beta_1(m))}{\beta_1(m)}=\left\{
\begin{aligned}
&0\qquad\qquad\qquad\qquad\qquad\qquad \hbox{if }\,\,0<m\le\frac{n-4}{n-2}\,\mbox{ and }n>4,\\
&2\big\{(n-4)-(n-2)m\big\}<0\,\quad \hbox{if }\,\, \frac{n-4}{n-2}< m<\frac{n-2}{n}.
\end{aligned}\right.
$$
Therefore (b), (c) and (d) follow.
\end{proof}

\begin{lemma} \label{lem-A1-A2-property}
Let  $n\geq3$ and $ 0< m <\frac{n-2}{n}$.   Then the following  holds. 

\begin{enumerate}[(a)]

\item   $A_1(m,\beta) $ is strictly increasing with respect to $\beta > \beta_0(m) $ if $ \frac{n-2}{n+2}\leq m <\frac{n-2}{n} $.  

\item  $A_2(m,\beta)$ is strictly decreasing with respect to $\beta >\beta_0(m) $  if $ \frac{n-2}{n+2}\leq m <\frac{n-2}{n} $.  

\item $A_1(m,\beta)>0$  if $ \frac{n-2}{n+2}\leq m <\frac{n-2}{n} $ and $\beta>\beta_1(m)$.

\item $A_1(m,\beta)<0$   if $\max\left( \frac{n-4}{n-2},\frac{n-2}{n+2}\right) \leq  m <\frac{n-2}{n} $ and $\beta_0(m)<\beta<\beta_1(m)$.

\end{enumerate}
\end{lemma}
\begin{proof}
By direct computation for any $i=1,2$,
\begin{align}\label{Ai-derivative-ineqn}
&2(1-m)\frac{\1}{\1\beta}A_i(m,\beta)\notag\\
=&-2(1-m)\gamma_i(m,\beta)-2(1-m)\beta\frac{\1}{\1\beta}\gamma_i(m,\beta)\notag\\
=&-2(1-m)\gamma_i(m,\beta)-\beta\left(1+\frac{(-1)^i A_0(m,\beta)}{\sqrt{A_0(m,\beta)^2-8(n-2-nm)(1-m)}} \right)\frac{\1}{\1\beta}A_0(m,\beta)\notag\\
=&\left(\sqrt{A_0(m,\beta)^2-8(n-2-nm)(1-m)} +(-1)^i A_0(m,\beta)\right)\cdot\notag\\
&\qquad\cdot\left((-1)^{i+1}-\frac{\beta \frac{\1}{\1\beta}A_0(m,\beta)}{\sqrt{A_0(m,\beta)^2-8(n-2-nm)(1-m)}} \right)\notag\\
=&\frac{\sqrt{A_0(m,\beta)^2-8(n-2-nm)(1-m)} +(-1)^i A_0(m,\beta) }{\sqrt{A_0(m,\beta)^2-8(n-2-nm)(1-m)}}\cdot\notag\\
&\qquad \cdot\left((-1)^{i+1}{\sqrt{A_0(m,\beta)^2-8(n-2-nm)(1-m)}} - 2\beta(n-2-nm)\right)\notag\\
=&\frac{\sqrt{A_0(m,\beta)^2-8(n-2-nm)(1-m)}+(-1)^i A_0(m,\beta) }{\sqrt{A_0(m,\beta)^2-8(n-2-nm)(1-m)}}\cdot\notag\\
&\quad\cdot\left({(-1)^{i+1}\sqrt{A_0(m,\beta)^2-8(n-2-nm)(1-m)}} -A_0(m,\beta)+(n-2)-(n+2)m\right).
\end{align}
Hence if $\frac{n-2}{n+2}\leq m <\frac{n-2}{n}$, then by \eqref{eq-A_0-pos} and \eqref{Ai-derivative-ineqn},
\begin{equation*}
\frac{\1}{\1\beta}A_1(m,\beta)>0>\frac{\1}{\1\beta}A_2(m,\beta)\qquad\forall \beta>\beta_0(m).
\end{equation*}
Thus (a) and  (b) follow. Utilizing  Lemma \ref{lem-A1-0} and (a), we get (c) and (d) and the lemma follows. 
\end{proof}

Based on Lemmas \ref{lem-A1-0} and \ref{lem-A1-A2-property}, we summarize the sign of $A_1=A_1(m,\beta)$  with respect to   $0<m<\frac{n-2}{n}$, $\beta>\beta_0(m)$, as follows, which plays a key role  in the asymptotic analysis of $w=w_\lambda$ near $s=\infty$. 

\begin{cor}\label{cor-sign-A12}
Let  $n\geq3$ and $ 0< m <\frac{n-2}{n}$.  Then the following  holds.
\begin{enumerate}[(a)]
\item $A_1(m,\beta) >0$ if \eqref{eq-cond-mono} holds.

\item $A_1(m,\beta) <0$ if \eqref{eq-cond-non-mono} holds. 
 
\item $A_1(m,\beta) =0$ and $A_2(m,\beta)<0$ if  \eqref{eq-cond-mono-beta1} holds.
 
\end{enumerate}
\end{cor} 

\begin{proof}
By  Lemmas \ref{lem-beta0-beta1}, \ref{lem-A1-0} and   \ref{lem-A1-A2-property}, the continuity of  $A_1(m,\beta)$  on the connected domain 
$$
\left\{(m,\beta):0<m<\frac{n-2}{n}, \beta>\beta_0(m)\right\}
$$
and the intermediate value theorem, we deduce that 
\begin{equation}\label{A1-positive}
A_1(m,\beta)>0\quad\mbox{ for any 
$\left\{\begin{aligned}
&0< m<\frac{n-2}{n}\\
&\beta>\beta_1(m)
\end{aligned}\right. 
\,\,\mbox{and}\,\,\left\{\begin{aligned}
&0<m<\frac{n-4}{n-2},\,\,n>4\\
&\beta_0(m)< \beta\leq \beta_1(m),
\end{aligned}\right.$}
\end{equation}
and
\begin{equation}\label{A1-positive2}
 A_1(m,\beta)<0\quad\mbox{ for any $\max\left(0,\frac{n-4}{n-2}\right)<m<\frac{n-2}{n}, \,\,\beta_0(m)<\beta<\beta_1(m)$}.
\end{equation}
By (d) of Lemma \ref{lem-A1-0},   \eqref{A1-positive} and \eqref{A1-positive2}, the corollary follows.
\end{proof}

By \eqref{eq-w-sol-lin-op},  Corollary \ref{cor-sign-A12},   and  an argument similar to the proof of Lemma 3.2 of \cite{DKS} and Section 3 of \cite{Hui}, we have the following lemma.

\begin{lemma}\label{lem-est-w}
Let  $n\geq3$, $ 0< m <\frac{n-2}{n}$, $\beta> \beta_0(m)$ and $\gamma_2=\gamma_2(m,\beta)>\gamma_1=\gamma_1(m,\beta)>0$  be  the two positive roots  of \eqref{eq-char} given by \eqref{eq-roots-gamma-intro}. Let     $w=w_\lambda$   be the solution of \eqref{eq-2nd-asymp-w-lin} which is given by \eqref{eq-w-sol-lin-op} and let 
\begin{equation}\label{I1-I2-defn}
{\bf I}_i:= \int_{-\infty}^\infty  e^{ \gamma_i t}  \Phi(t)\,dt\qquad\forall i=1,2.
\end{equation}
Then there exists a constant $C>0$ such that
\begin{equation}\label{w-upper-bd}
|w(s)|\le C e^{-\gamma_1 s } \qquad\forall s\in\R.
\end{equation} 
Moreover the following holds.

\begin{enumerate}[(a)]

\item  If  $A_1=A_1(m,\beta)>0$, then 
\begin{equation}\label{w-negative}
w(s)<0\qquad\forall s\in\R. 
\end{equation}

\item If  $A_1=A_1(m,\beta)<0$, then there exists a constant $s_*>s_0$  such that  
\begin{equation}\label{w-positive-near-infty}
 w(s)>0\qquad\forall s\geq s_*. 
\end{equation} 
\item If $A_1(m,\beta)\ne 0$, then
\begin{equation}\label{eq-est-I1-bound}
0<{\bf I}_1<\infty.
\end{equation}

 \item     If  $A_1=A_1(m,\beta)=0$ and $A_2=A_2(m,\beta)<0$, then \eqref{w-negative} holds and
\begin{equation}\label{eq-est-I2-bound}
 0<  {\bf I}_2<\infty.
 \end{equation}
 \end{enumerate}
Here $s_0>0$ is the constant  appearing in  \eqref{w-small-at-infty}.
\end{lemma}
\begin{proof} 
Since the proof of this lemma is  similar to the proof of Lemma 3.2 of \cite{DKS} and Section 3 of \cite{Hui}, we will only sketch the proof here. Multiplying \eqref{eq-w-linearized} by $e^{\gamma_1t}$ and $e^{\gamma_2t}$ respectively  and integrating over $(-\infty,s)$, 
\begin{equation}\label{eq-w_s-w-2}
w' +\gamma_2 w= - C_*A_1 e^{-\gamma_1s} \int_{-\infty}^s e^{\gamma_1t}\Phi(t)dt-C_*\beta\Phi(s)\quad\forall s\in\R 
\end{equation} 
and
\begin{equation}\label{eq-w_s-w-2a}
w' +\gamma_1 w= - C_*A_2 e^{-\gamma_2s} \int_{-\infty}^s e^{\gamma_2t}\Phi(t)dt-C_*\beta\Phi(s)\quad\forall s\in\R. 
\end{equation}
Note that since $\gamma_2>\gamma_1>0$, we have
 \begin{equation}\label{eq-est-gamma-st}
e^{-\gamma_1 (s-t)}> e^{-\gamma_2(s-t)} \quad\forall  s > t.
 \end{equation}
We now divide the proof into two cases.

\noindent{\bf Case 1}: $A_1\ne 0$.

\noindent{\bf Case 1a}: $A_1>0$. 

\noindent By \eqref{eq-w-sol-lin-op}, \eqref{eq-a1>a2}  and \eqref{eq-est-gamma-st},  we get \eqref{w-negative} and
\begin{equation}\label{eq-w-exp}
 |w(s)| = -w(s)\leq C_0\left(A_1+|A_2|\right)e^{-\gamma_1 s}\int_{-\infty}^s e^{\gamma_1 t}\,\Phi(t)dt\qquad\forall s\in\R. 
\end{equation}

\noindent{\bf Case 1b}: $A_1<0$. 

\noindent Let $C_3=C_*|A_1|\int_{-\infty}^{s_0} e^{\gamma_1 t}\Phi(t)\,dt$. Then $C_3>0$. 
By \eqref{delta2-defn}, \eqref{w-small-at-infty}, \eqref{eq-Phi-w^2} and \eqref{eq-w_s-w-2}, 
\begin{align}\label{w'-ineqn10}
&w' +\gamma_2 w\ge C_3e^{-\gamma_1s}-c_2C_*\beta w^2\qquad\forall s\ge s_0\notag\\
\Rightarrow\quad&\frac{d}{ds}\left(e^{\gamma_2s+c_2C_*\beta\int_{s_0}^sw(t)\,dt}w(s)\right)\ge C_3e^{(\gamma_2-\gamma_1)s+c_2C_*\beta\int_{s_0}^sw(t)\,dt}\ge C_3e^{\frac{(\gamma_2-\gamma_1)}{2}s}\qquad \forall s\ge s_0\notag\\ 
\Rightarrow\quad&e^{\gamma_2s+C_4\int_{s_0}^sw(t)\,dt}w(s)\ge \frac{2C_3}{\gamma_2-\gamma_1}\left(e^{\frac{(\gamma_2-\gamma_1)}{2}s}-e^{\frac{(\gamma_2-\gamma_1)}{2}s_0}\right)+e^{\gamma_2s_0}w(s_0)\qquad \forall s\ge s_0 
\end{align} 
where $C_4=c_2C_*\beta$.
Since $\gamma_2>\gamma_1>0$, by \eqref{w'-ineqn10} there exists $s_*>s_0$ such that 
\eqref{w-positive-near-infty} holds. Since $A_2<A_1<0$, by \eqref{eq-w-sol-lin-op}, it follows that 
\begin{equation}\label{eq-w-exp-1}
0<w(s)\leq C_0|A_1|e^{-\gamma_1 s}\int_{-\infty}^s e^{\gamma_1 t}\Phi(t)dt\qquad\forall s\ge s_*.
\end{equation}

\noindent Hence when $A_1\ne 0$, by \eqref{eq-w-exp} and \eqref{eq-w-exp-1} there exist constants $s_*>s_0$ and $C>0$ such that 
\begin{equation}\label{eq-w-exp-decay}
|w(s)|\le Ce^{-\gamma_1 s}\int_{-\infty}^s e^{\gamma_1 t}\,\Phi(t)dt\qquad\forall s\ge s_*,
\end{equation} 
 and $w(s)>0$ for any $s\ge s_*$ if $A_1<0$.

Now we claim that \eqref{eq-est-I1-bound} holds. 
 Suppose to the contrary that 
\begin{equation}\label{I1=infty}
{\bf I}_1=\infty.
\end{equation}
Let
\begin{equation*}
v(s)=\left\{\begin{aligned}
&-w(s)\quad\forall s\in\R\quad\mbox{ if }A_1>0,\\
&w(s)\qquad\,\,\forall s\in\R\quad\mbox{ if }A_1<0,
\end{aligned}\right.
\end{equation*}
and
$$
J_1(s)= \int_{s_0}^s e^{\gamma_1 t}\,\Phi(t)dt,
$$ 
where $s_0$ is as given by \eqref{eq-Phi-w^2}.
Then by  \eqref{eq-Phi-w^2}, \eqref{eq-w-exp-decay}, \eqref{I1=infty} and an argument similar to the proof of Lemma 3.2 of \cite{DKS} there exists constants $s_1>s_*$, and  $c_3>0$ such that
\begin{equation}\label{eq-J1-exp}
J_1(s)\geq c_3 e^{\gamma_1 s}\qquad \forall s\ge s_1.
\end{equation}
By \eqref{eq-Phi-w^2}, \eqref{eq-w_s-w-2} and    \eqref{eq-J1-exp},  we deduce that
\begin{align}\label{w-ineqn5}
&v'+\gamma_2v\ge C_*|A_1|e^{-\gamma_1s} J_1(s)-c_2C_*\beta w^2\qquad \forall s\ge s_1\notag\\
\Rightarrow \q&v'+\gamma_2v\ge c_3C_*|A_1|-c_2C_*\beta w^2\qquad \forall s\ge s_1.
\end{align} 
Since $\displaystyle\lim_{s\to\infty}w(s)=0$, there exists $s_2>s_1$ such that 
\begin{equation}\label{w-upper-bd5}
|w(s)|\le\min\left(\frac{c_3C_*|A_1|}{4\gamma_2},\sqrt{\frac{c_3|A_1|}{4c_2\beta}}\right)\qquad\forall s\ge s_2.
\end{equation}
By using \eqref{w-ineqn5} and \eqref{w-upper-bd5}, it follows that 
\begin{align*}
&v'(s)\ge\frac{c_3C_*|A_1|}{2}\qquad\forall s\ge s_2\\
\Rightarrow\quad&v(s)=|w(s)|\to\infty\qquad\mbox{ as }\, s\to\infty,
\end{align*}
which contradicts the fact that $\displaystyle\lim_{s\to\infty}w(s)=0$. Hence \eqref{I1=infty} does not hold and \eqref{eq-est-I1-bound} follows. 
 
\noindent{\bf Case 2}: $A_1 =0$ and $A_2<0$. 

\noindent By \eqref{eq-w-sol-lin-op}, we have 
 \begin{equation}\label{eq-w-sol-lin-op-exp-2}
w(s)=-C_0|A_2 |e^{-\gamma_2s}\int_{-\infty}^se^{ \gamma_2 t} \Phi(t)\,dt<0\qquad\forall s\in\R
\end{equation}  
and \eqref{w-negative} follows.
Utilizing \eqref{eq-w_s-w-2a}, \eqref{eq-w-sol-lin-op-exp-2} and an argument  similar to the proof  of \eqref{eq-est-I1-bound}  in  Case 1, we get  \eqref{eq-est-I2-bound}. 

Finally by  \eqref{eq-est-I1-bound}, \eqref{eq-est-I2-bound}, \eqref{eq-w-exp-decay} and \eqref{eq-w-sol-lin-op-exp-2}, we get \eqref{w-upper-bd} and the lemma follows.
\end{proof}
 
\begin{lemma}\label{w-limit-behavior-lem}
Let  $n\geq3$, $ 0< m <\frac{n-2}{n}$, $\beta> \beta_0(m)$ and $\gamma_2=\gamma_2(m,\beta)>\gamma_1=\gamma_1(m,\beta)>0$  be  the two positive roots  of \eqref{eq-char} given by \eqref{eq-roots-gamma-intro}. Let     $w=w_\lambda$   be the solution of \eqref{eq-2nd-asymp-w-lin} which is given by \eqref{eq-w-sol-lin-op} and ${\bf I}_1$, ${\bf I}_2$ be given by \eqref{I1-I2-defn}. Then the following  holds.
\begin{enumerate}[(a)]
\item  If $A_1=A_1(m,\beta)\not=0$, then 
$$
\lim_{s\to\infty} e^{ \gamma_1s} w(s)  = -C_0 A_1\,{\bf I}_1. 
$$ 

\item  If  $A_1=A_1(m,\beta)=0$ and $A_2=A_2(m,\beta)<0$, then 
$$
\lim_{s\to\infty} e^{ \gamma_2s} w(s)  = C_0 A_2\,{\bf I}_2 <0.
$$

\end{enumerate}
\end{lemma}
\begin{proof}
We first observe that  (b)  follows immediately from  \eqref{eq-w-sol-lin-op-exp-2} and (d) of Lemma \ref{lem-est-w}. Hence it remains to prove (a). Suppose now that $A_1(m,\beta)\not=0$.  By \eqref{eq-w-sol-lin-op}  and (c) of Lemma \ref{lem-est-w}, it suffices to prove that 
\begin{equation}\label{eq-phi-2-limit}
\lim_{s\to\infty} e^{-(\gamma_2-\gamma_1)s}\int_{-\infty}^s  e^{\gamma_2t }\Phi(t)dt \,=0.
\end{equation}
By \eqref{eq-Phi-w^2} and \eqref{w-upper-bd},  we deduce that 
\begin{align*}
e^{-(\gamma_2-\gamma_1)s}\int_{s_0}^s  e^{\gamma_2t }\Phi(t) dt&\le C_1\,e^{-(\gamma_2-\gamma_1)s}\int_{s_0}^s e^{(\gamma_2-2\gamma_1)t}\,dt\qquad \forall s\ge s_0\\
&\le C_1' e^{-(\gamma_2-\gamma_1)s} \left(s+ {e^{ (\gamma_2-2\gamma_1)s}} \right)\qquad\forall s\ge s_0\\ 
&=C_1' \left(se^{-(\gamma_2-\gamma_1)s}+ {e^{-\gamma_1s}} \right)\qquad\forall s\ge s_0\\ 
&\to 0\qquad\mbox{ as }\,\,s\to\infty
\end{align*} for some constants $C_1>0$, $C_1'>0$,
and  \eqref{eq-phi-2-limit} follows. 
\end{proof}

\begin{proof}[\textbf{Proof of Theorem \ref{thm-self-sol-2nd-asymp}}]
By  Corollary \ref{cor-sign-A12}, Lemma \ref{w-limit-behavior-lem} and the relation between $w=w_{\lambda}$ and $f_{\lambda}$,   for any given $(m,\beta)$ there exists a constant $B_\lambda>0$ such that 
\begin{equation}\label{eq-f-conv-est-B}
\left\{\begin{aligned}
&\lim_{r\to\infty} r^{\gamma} \left\{\left[\left({ {r^2}/{C_*}}\right)^{\frac{1}{1-m}} f_\lambda(r)\right]^m-1\right\}=- mB_\lambda\quad \hbox{ if \eqref{eq-cond-mono} holds;}\\
&\lim_{r\to\infty} r^{\gamma} \left\{\left[\left({ {r^2}/{C_*}}\right)^{\frac{1}{1-m}} f_\lambda(r)\right]^m-1\right\}=  mB_\lambda\quad\,\,\,\,\hbox{ if \eqref{eq-cond-non-mono} holds;}\\
\end{aligned} \right.
\end{equation} 
with $\gamma= \gamma_1(m,\beta)$, and 
\begin{equation}\label{eq-f-conv-est-B-2}
\lim_{r\to\infty} r^{\gamma} \left\{\left[\left({ {r^2}/{C_*}}\right)^{\frac{1}{1-m}} f_\lambda(r)\right]^m-1\right\}=- mB_\lambda \quad\hbox{ if \eqref{eq-cond-mono-beta1} holds} 
\end{equation} 
with $\gamma= \gamma_2(m,\beta)$. 
Moreover by Lemma \ref{lem-scaling} for any $\lambda>0$,
\begin{align}\label{scaling-eqn}
&\lim_{r\to\infty}r^{\gamma} \left\{\left[\left({ {r^2}/{C_*}}\right)^{\frac{1}{1-m}} f_\lambda(r)\right]^m-1\right\}
=\lambda^{-\gamma}\lim_{\rho\to\infty}\rho^{\gamma} \left\{\left[\left({{\rho^2}/{C_*}}\right)^{\frac{1}{1-m}} f_1(\rho)\right]^m-1\right\}\notag\\
\Rightarrow\quad&\, B_{\lambda}= B_1\lambda^{-\gamma }.
\end{align}
Letting $B=B_1$,  by \eqref{eq-f-conv-est-B} and \eqref{scaling-eqn}, (a) and (b) follow. In light of  \eqref{eq-f-conv-est-B-2} and \eqref{scaling-eqn}, we will prove that $\gamma= \gamma_2(m,\beta)=2$ in case (c)  to complete the proof. 
 
Suppose now \eqref{eq-cond-mono-beta1} holds. Then $A_0(m,\beta_1(m))=n-2m-nm$, and  \eqref{eq-char} is equal to
\begin{align*}
&\gamma^2-\frac{n-2m-nm}{1-m}\gamma+\frac{2(n-2-nm)}{1-m}=0\\
\Rightarrow\quad&\gamma_2=2\,>\,\gamma_1=\frac{n-2-nm}{1-m}.
\end{align*}
Therefore (c) follows.
\end{proof}

\begin{cor}
Let  $n\geq3$, $ 0< m <\frac{n-2}{n}$ and $\beta> \beta_0(m)$. Suppose  either \eqref{eq-cond-mono} or \eqref{eq-cond-mono-beta1} holds. Then  for any $\lambda>0$,
\begin{equation}\label{eq-est-f-ld-tildeC}
f_{\lambda}(r) <\cC(r)=  \left(\frac{C_*}{ r^2}\right)^{ \frac{1}{1-m}} \qquad \forall r>0.
\end{equation}
\end{cor}
\begin{proof}
By Corollary \ref{cor-sign-A12}  and Lemma \ref{lem-est-w}, it holds that $w_\lambda(s)<0$ for all $s\in \R$. Hence
$$
\left({{r^2}/{C_*}}\right)^{\frac{1}{1-m}} f_\lambda(r)  =  g_\lambda^{1/m}(s)=\big\{1+w_\lambda(s)\big\}^{1/m}  <1\qquad  \forall r=e^{s}>0
$$
and \eqref{eq-est-f-ld-tildeC} follows.
\end{proof} 
 
\subsection{Properties of the self-similar profile $f_{\lambda}$}
Next we study properties of self-similar profiles $f_{\lambda}$ in  the cases of \eqref{eq-cond-mono}, \eqref{eq-cond-non-mono} and \eqref{eq-cond-mono-beta1}. Here and below $f_{\lambda}$ is the unique radially symmetric  smooth solution to  \eqref{eq-ellip}  given  by Theorem \ref{thm-cigar-soliton-existence}.

\begin{lemma}[{Monotonicity}]\label{lem-monotonicity}
Suppose either  \eqref{eq-cond-mono} or \eqref{eq-cond-mono-beta1} holds. Then 
\begin{equation}\label{f-lambda-derivative-positive16}
\frac{d f_{\lambda} }{d\lambda }(r)>0\qquad \forall r\geq 0,\,\, \lambda>0,
\end{equation}
and \eqref{eq-(A)-mono} holds.
\end{lemma}
\begin{proof}
By  Corollary \ref{cor-sign-A12} and  Lemma \ref{lem-est-w},  we have 
\begin{align} 
&g_\lambda(s)=\left({  e^{2s}/{C_*}}\right)^{\frac{m}{1-m}}  f_\lambda^m(e^s)  = 1+w_\lambda(s)<1\qquad\qquad\qquad\quad \forall s\in\R\label{eq-q<1}\\    
\Rightarrow\quad &g_\lambda'(s)=mC_*^{-\frac{m}{1-m}} e^{\frac{2ms}{1-m}} f_{\lambda}^{m-1}(e^s) \left(\,\frac{2}{1-m} f_\lambda(e^s) +e^sf'_{\lambda} (e^s)\,\right)\quad\,\, \forall s\in\R. \label{eq-est-q_s}
\end{align}
Hence  $g_\lambda'(s)>0$ for   sufficiently small $s\ll-1$ since $f_{\lambda}(0)=\lambda^{\frac{2}{1-m}}$ and $f_{\lambda}'(0)=0$.   
Now we claim that for $g=g_\lambda$,
\begin{equation}\label{q-s-derivative>0}
g'(s)>0\qquad\forall s\in\R.
\end{equation} 
Suppose  the claim \eqref{q-s-derivative>0} does not hold. Then  there exists  a constant $s_0'\in\R$ such that $g'(s_0')\le 0$. Hence by the intermediate value theorem there exists a constant $s_1\le s_0'$ such that $g'(s_1)=0$. Let 
$$
s_2:=\sup\left\{s'\in\R:\, g'(s)>0\quad\forall s<s' \right\}.
$$ 
Then $-\infty<s_2\le s_1$,  $g'(s_2)=0$ and $g''(s_2)\leq0$. Hence the equation \eqref{eq-2nd-asymp-q} for $g=g_\lambda$ evaluated at $s_2$  yields that 
\begin{equation*} 
0\geq g''+\left\{\frac{n-2-(n+2)m}{1-m}+\frac{\beta C_*}{m} g^{\frac{1}{m}-1}\right\}g'=\frac{2m(n-2-nm)}{(1-m)^2}\left(g -g^{\frac{1}{m}}\right) .
\end{equation*}
However by \eqref{eq-q<1} the right hand side of the above equation is positive and a contradiction arises. Thus the claim  \eqref{q-s-derivative>0} holds.  
 
By  \eqref{eq-est-q_s} and \eqref{q-s-derivative>0}, it holds that 
$$
\frac{2}{1-m} f_\lambda(r) +    r f'_{\lambda} (r) >0\qquad \forall r\ge 0
$$ 
since $f_{\lambda}(0)=\lambda^{\frac{2}{1-m}}$ and $f_{\lambda}'(0)=0$. 
Therefore, it follows from Lemma \ref{lem-scaling} that for any $\lambda>0$,
\begin{equation*} 
\frac{d f_{\lambda}}{d\lambda }(r) =\frac{d}{d\lambda }\left\{ \lambda^{\frac{2}{1-m}} f_1(\lambda r)\right\}=\lambda^{\frac{2}{1-m}-1}   \left\{\frac{2}{1-m} f_1(\lambda r) + \lambda   r \cdot f'_{1} ( \lambda r)  \right\}>0\qquad\forall r\ge 0,
\end{equation*}
and \eqref{f-lambda-derivative-positive16} follows. Integrating \eqref{f-lambda-derivative-positive16} over $\lambda\in(\lambda_1,\lambda_2)$,  we deduce \eqref{eq-(A)-mono} finishing the proof.
\end{proof}

With the  non-monotone   condition   \eqref{eq-cond-non-mono}, we have a reverse monotonicity of $f_{\lambda}(r)$ with respect to $\lambda>0$  for sufficiently large $r\gg1.$

\begin{lemma}[Reverse monotonicity near infinity]\label{lem-rev-monotonicity}
Suppose  \eqref{eq-cond-non-mono} holds. Then there exists a constant $R_0>0$  such that   
\begin{equation}\label{f-lambda-derivative<0} 
\frac{d f_{\lambda}}{d\lambda }(r)<0\qquad \forall r\ge R_0/\lambda,\,\,\lambda>0,
\end{equation} 
and \eqref{eq-(B)-non-mono} holds.
\end{lemma}
\begin{proof} 
  Let $\lambda>0$. We first observe that by Corollary \ref{cor-sign-A12} and  Lemma \ref{lem-est-w} there exists a constant $s_*>s_0$ such that
\begin{equation}\label{eq-q>1}
g_\lambda(s)=1+w_\lambda(s)>1\qquad\forall s>s_*.
\end{equation}
Since $\displaystyle \lim_{s\to\infty}g_\lambda(s)=1$ by  \eqref{eq-2nd-asymp-q},  there exists a constant  $s_1=s_1(\lambda)>s_*$ such that  
 $g'(s_1)<0$.  We claim that  $g=g_\lambda$ satisfies
\begin{equation}\label{q-s-derivative<0}
g'(s)<0\qquad\forall s\ge s_1.
\end{equation} 
Suppose to the contrary that the claim does not hold. Then there exists $s_2>s_1$ such that $g'(s)<0$ for any $s_1\le s<s_2$ and 
 $g'(s_2)=0$. Hence  $g''(s_2)\ge 0$. Thus the equation \eqref{eq-2nd-asymp-q} for $g=g_\lambda$ evaluated at $s_2$ implies that 
\begin{equation*} 
0\leq g''+\left\{\frac{n-2-(n+2)m}{1-m}+\frac{\beta C_*}{m} g^{\frac{1}{m}-1}\right\}g'=\frac{2m(n-2-nm)}{(1-m)^2}\left(g-g^{\frac{1}{m}}\right) .
\end{equation*}
On the other hand the right hand side of the above equation is negative by  \eqref{eq-q>1}. Hence a contradiction arises and  the claim \eqref{q-s-derivative<0} follows. 
 
By \eqref{f-lambda-f1-eqn}, \eqref{eq-est-q_s} and \eqref{q-s-derivative<0}, there exists a constant $s_1(1)$ for $g=g_1$ such  that 
\begin{align*}
&g_1'(s)<0\qquad\forall s\ge s_1(1)\\
\Rightarrow\quad&\frac{2}{1-m}f_1(r)+rf'_{1}(r)<0\qquad \forall r\ge R_0:=e^{s_1(1)}\\
\Rightarrow\quad&\frac{d f_{\lambda}}{d\lambda }(r) =\lambda^{\frac{2}{1-m}-1}\left\{\frac{2}{1-m}f_1(\lambda r)+\lambda r \cdot f'_{1}( \lambda r)\right\}{ < 0}\qquad\forall  r\ge R_0/\lambda, \,\,\lambda>0,
\end{align*}
and \eqref{f-lambda-derivative<0}  follows. \eqref{eq-(B)-non-mono} then follows by integrating \eqref{f-lambda-derivative<0} with respect to $\lambda$ over $(\lambda_1,\lambda_2)$.
\end{proof}

\begin{cor} Suppose \eqref{eq-cond-non-mono} holds. Then for any  $\lambda_2>\lambda_1>0$, there exist  constants $r_2>r_0>r_1>0$ such that $f_{\lambda_1}(r_0)=f_{\lambda_2}(r_0)$ and
\begin{enumerate}[(i)]
\item $f_{\lambda_2}(r)>f_{\lambda_1}(r)\qquad\forall 0\leq r< r_1$,
\item $f_{\lambda_2}(r)<f_{\lambda_1}(r)\qquad\forall r>r_2$.
\end{enumerate}
\end{cor} 
\begin{proof}
Since $f_{\lambda_2}(0)=\lambda_2^{\frac{2}{1-m}}>\lambda_1^{\frac{2}{1-m}}
=f_{\lambda_1}(0)$, by continuity there exists $r_1>0$ such that (i) holds. By Lemma \ref{lem-rev-monotonicity} there exists  a constant $r_2>r_1>0$ such that (ii)  holds true. By  (i), (ii) and the intermediate value theorem, there exists a constant $r_0\in (r_1,r_2)$ such that $f_{\lambda_1}(r_0)=f_{\lambda_2}(r_0)$ and the corollary follows.
\end{proof} 

\begin{lemma}\label{lem-diff-solitons-integrable}
Let $\lambda_2>\lambda_1>0$. Then the following holds.

\begin{enumerate}[(a)]
 
\item 
If \eqref{eq-cond-mono} holds, then   
\begin{equation}\label{eq-exp-integ10}
n-\frac{2}{1-m}-\gamma_1(m,\beta)>0
\end{equation} 
and $f_{\lambda_1}-f_{\lambda_2}\not\in L^1(\R^n)$. 

\item 
If either \eqref{eq-cond-non-mono} or \eqref{eq-cond-mono-beta1} holds, then 
$f_{\lambda_1}-f_{\lambda_2}\in L^1(\R^n)$. 

\end{enumerate}
\end{lemma}
\begin{proof}
Suppose \eqref{eq-cond-mono}, \eqref{eq-cond-non-mono} or \eqref{eq-cond-mono-beta1} holds. By Theorem \ref{thm-self-sol-2nd-asymp} there is  a constant $B\ne 0$ such that 
\begin{equation*}
f_1(r)= \left(\frac{C_*}{ r^2}\right)^{ \frac{1}{1-m}}\Big\{\,1+  Br^{-\gamma}+o (r^{-\gamma})\,\Big\}\quad\mbox{ as }r=|y|\to\infty 
\end{equation*}
where $\gamma=\gamma_1(m,\beta)$ if \eqref{eq-cond-mono} or \eqref{eq-cond-non-mono} holds, and $\gamma=\gamma_2(m,\beta_1(m))=2$ if \eqref{eq-cond-mono-beta1} holds.
Hence for each $0<\ve<1$,  there exists a  constant    $R_\ve>0$   such that
\begin{align*}
&\left(\frac{C_*}{r^2}\right)^{\frac{1}{1-m}}\Big\{1+  \big(B-\ve{|B|}  \big)r^{-\gamma} \Big\}\leq f_{1}(r)\leq \left(\frac{C_*}{r^2}\right)^{\frac{1}{1-m}}\Big\{1+\big(B+\ve{|B|}  \big)r^{-\gamma} \Big\}\quad\forall r\geq R_\ve. 
\end{align*}
 This together with \eqref{f-lambda-f1-eqn} implies that 
\begin{align*}
\left(\frac{C_*}{r^2}\right)^{\frac{1}{1-m}}\Big\{1+ \big(B-\ve{|B|}  \big)\ld^{-\gamma}r^{-\gamma}\Big\}&\le f_{\ld}(r)\\
&\leq \left(\frac{C_*}{r^2}\right)^{\frac{1}{1-m}}\Big\{1+ \big(B+\ve{|B|}  \big)\ld^{-\gamma}r^{-\gamma}\Big\}\quad\forall r\geq R_\ve/\ld,\,\, \lambda>0.
\end{align*}
Hence by choosing a sufficiently small $\ve>0$, there exist constants $R_1$ and $C_2>C_1>0$ such that
\begin{equation}\label{eq-est-L1-diff}
 C_1\int_{R}^\infty r^{n-1 -\frac{2}{1-m}-\gamma}dr\le\int_{|y|\geq R}\left|f_{\lambda_1}-f_{\lambda_2}\right| dy \leq C_2\int_{R}^\infty r^{n-1 -\frac{2}{1-m}-\gamma}dr\quad\forall  R\geq R_1/\lambda_1.
\end{equation} 
  Note that   for $j=1,2$,
 \begin{equation}\label{eq-exp-integ}
 n-\frac{2}{1-m}-\gamma_j(m,\beta) = \frac{1}{(1-m)\beta_1(m)}-\gamma_j(m,\beta).
 \end{equation} 
We now divide the proof into four cases.

\noindent{\bf Case 1}: (i) of \eqref{eq-cond-mono} holds. 

\noindent Note that $\gamma=\gamma_1(m,\beta)$ in \eqref{eq-est-L1-diff} by   Theorem \ref{thm-self-sol-2nd-asymp}. Then by  (i) of \eqref{eq-cond-mono}, \eqref{eq-def-A_j}  and   Corollary \ref{cor-sign-A12}, we have
 $$
\frac{1}{(1-m)\beta _1(m)}-\gamma_1(m,\beta)\geq  \frac{1}{(1-m)\beta}-\gamma_1(m,\beta) =  \frac{1}{\beta}\, A_1(m,\beta)>0 
$$
since $\beta>\beta_1(m)$. This estimate together with \eqref{eq-exp-integ} implies that \eqref{eq-exp-integ10} holds and the integral on the left hand side of \eqref{eq-est-L1-diff} is  infinite. Hence  $f_{\lambda_1}-f_{\lambda_2}\not\in L^1(  \R^n) $. 

\noindent{\bf Case 2}: (ii) of \eqref{eq-cond-mono} holds. 

\noindent By Theorem \ref{thm-self-sol-2nd-asymp}, we note that $\gamma=\gamma_1(m,\beta)$ in \eqref{eq-est-L1-diff}.  By \eqref{eq-def-A_0(be)} and (ii) of \eqref{eq-cond-mono}, it follows that
\begin{align*}
 2(n-2-nm) -A_0(m,\beta)&= n-2-(n-2)m-2(n-2-nm)\beta\\
 &\geq  n-2-(n-2)m-2(n-2-nm)\beta_1(m)\\
 &=  n-4-(n-2)m.
\end{align*}
Hence by \eqref{eq-roots-gamma-intro} and (ii) of \eqref{eq-cond-mono}, we deduce that
\begin{align*}
 &n-\frac{2}{1-m}-\gamma_1(m,\beta) \\
 =&\frac{1}{2(1-m)} \left\{ 2(n-2-nm) -A_0(m,\beta) +  \sqrt{A_0(m,\beta)^2-8(n-2-nm)(1-m)}\right\}\\
 =&\frac{1}{2(1-m)} \left\{ n-4-(n-2)m +  \sqrt{A_0(m,\beta)^2-8(n-2-nm)(1-m)}\right\}\\
 >&0 ,
 \end{align*}
which  implies that \eqref{eq-exp-integ10} holds and the integral on the left hand side  of \eqref{eq-est-L1-diff} is infinite. Hence $f_{\lambda_1}-f_{\lambda_2}\not\in L^1(\R^n)$. 

\noindent{\bf Case 3}: \eqref{eq-cond-non-mono} holds. 

\noindent By  \eqref{eq-cond-non-mono}, \eqref{eq-def-A_j}  and   Corollary \ref{cor-sign-A12}, we have 
 $$
\frac{1}{(1-m)\beta _1(m)}-\gamma_1(m,\beta)<\frac{1}{(1-m)\beta}-\gamma_1(m,\beta)=\frac{ A_1(m,\beta)}{\beta}<0 .
$$
This together with \eqref{eq-exp-integ} implies that  the integral on the right hand side  of \eqref{eq-est-L1-diff} is  finite.  Hence  $f_{\lambda_1}-f_{\lambda_2}\in L^1(\R^n)$. 
 
 \noindent{\bf Case 4}:  \eqref{eq-cond-mono-beta1} holds. 

\noindent Since $\gamma=2$  by Theorem \ref{thm-self-sol-2nd-asymp},    \eqref{eq-cond-mono-beta1} implies
$$ 
n-\frac{2}{1-m}-\gamma=\frac{ n-4-(n-2)m}{1-m}<0  .
$$
This yields that the integral on the right hand side  of \eqref{eq-est-L1-diff} is  finite.  Hence  $f_{\lambda_1}-f_{\lambda_2}\in L^1(\R^n)$ and the lemma follows.  
\end{proof}

\begin{proof}[\textbf{Proof of Theorem \ref{thm-self-sol-2nd-asymp-mono}}] 
Theorem \ref{thm-self-sol-2nd-asymp-mono} follows directly from  Lemmas \ref{lem-monotonicity},   \ref{lem-rev-monotonicity} and   \ref{lem-diff-solitons-integrable}. 
\end{proof}

\section{Existence of  solutions to the fast diffusion equation}\label{sec-exist}
 
In this section we will establish the existence of solutions to the fast diffusion equation    \eqref{eq-fde}   based on the result of \cite{Hsu2}, provided that  the initial value $u_0$  is   close in $L^1(\R^n)$ to  the initial value  of  a   self-similar solution to \eqref{eq-fde}.    Let  $n\geq 3$,  $0<m< \frac{n-2}{n}$ and $ \beta \geq  \beta_e(m)=\frac{m}{n-2-nm} $ with $\alpha=   \frac{2\beta+1}{1-m}$. We recall that  for any $T>0$ and  $\lambda>0$,  a self-similar solution $U_\lambda$  to   \eqref{eq-fde} in $\R^n\times(0,T)$   is given by 
$$
U_\lambda(x,t)=(T-t)^\alpha f_{\lambda}\left((T-t)^\beta x\right)\qquad\forall (x,t)\in\R^n\times(0,T) 
$$  
where   $ f_{\lambda}$  is the solution of \eqref{eq-ellip}. 

\begin{proof} [\textbf{Proof of Theorem \ref{thm-existence-fde}}]
 \noindent{\bf Step 1.}   In order to prove   the existence of  the short time  solution to  \eqref{eq-fde} in $\R^n\times(0,T_0)$ for some constant $T_0>0$, we will apply Theorem 1.1 of  \cite{Hsu2} using  the initial assumption  \eqref{eq-initial-soliton-L1}.   By  \eqref{eq-initial-soliton-L1},  it follows that 
\begin{equation}\label{eq-existence-sol-short-time}
\begin{split}
\liminf_{R\to\infty}\frac{1}{R^{n-\frac{2}{1-m}}}\int_{|x|\leq R} u_0(x)dx&= \liminf_{R\to\infty}\frac{1}{R^{n-\frac{2}{1-m}}}\int_{|x|\leq R} U_{\lambda_0}(x,0)dx\\
&=T^{ \frac{1}{1-m} }\cdot \liminf_{R\to\infty}\frac{1 }{(T^\beta R)^{n-\frac{2}{1-m}}}\int_{|y|\leq T^{\beta}R} f_{\lambda_0}(y)dy.
\end{split}
\end{equation} 
Here we used that $0<m< \frac{n-2}{n}$, and $\alpha=\frac{2\beta+1}{1-m}$. By \eqref{eq-1st-asymp-ellip} 
there exists a   constant $c_0>0$ such that   
\begin{equation}\label{eq-est-ex-C_0}
\liminf_{r\to\infty} \frac{1}{r^{n-\frac{2}{1-m}}}\int_{|y|\leq r} f_{\lambda_0}(y)dy\geq c_0>0.
\end{equation}
Let  $C_1>0$ be the constant appearing   in Theorem 1.1 of \cite{Hsu2}. By 
\eqref{eq-existence-sol-short-time} and \eqref{eq-est-ex-C_0}, we have 
$$
\liminf_{R\to \infty}\frac{1}{R^{n-\frac{2}{1-m}}}\int_{|x|\leq R} u_0(x)dx\,\geq c_0T^{\frac{1}{1-m}}=C_1 T_0^{\frac{1}{1-m}}
$$
with a constant $T_0= (c_0/C_1)^{1-m}\, T>0$. Therefore by \cite[Theorem 1.1]{Hsu2}  there   exists a unique positive solution    $u\in C^\infty\left(\R^n\times(0,T_0)\right)$ of  \eqref{eq-fde}  in $\R^n\times(0,T_0)$ which satisfies \eqref{ab-ineqn}   in $ \R^n\times(0,T_0) $. 

 Let $T_*>0$ be the maximal existence time of the   positive solution    $u\in C^\infty\left(\R^n\times(0,T_*)\right)$  to \eqref{eq-fde} in $ \R^n\times(0,T_*) $ which satisfies \eqref{ab-ineqn} in $\R^n\times(0,T_*)$.
 Note that $T_*\geq T_0>0.$ 
  By utilizing the initial condition  \eqref{u0<U-lambda_2},   the approximating procedure for the   the construction  of the solution $u$ in the proof of     \cite[Theorem 1.1]{Hsu2} yields that  \eqref{u-u-lambda-ineqn2} holds in $\R^n\times(0,T_*)$.    
   In particular,  we deduce from  \eqref{u-u-lambda-ineqn2} that  $0< T_*\leq T.$  Furthermore,  
if $u_0$ also satisfies \eqref{u0>U-lambda}, then 
 by the construction of the solution of \eqref{eq-fde}  in \cite{Hsu2}, \eqref{u>u-lambda-ineqn1} holds in $\R^n\times(0,T_*)$.   

\medskip
 
\noindent{\bf Step 2.} 
Next we will   establish   the $L^1$-contraction principle  \eqref{eq-exist-sol-contraction}   for any $0<t<T_*$. We first observe that by \eqref{eq-1st-asymp-ellip}, for any $0<S<T_*$ there exists a   constant $c_S>0$   such that 
$$
U_{\lambda_0}(x,t)\geq  c_{S}\min\left(1, |x|^{-\frac{2}{1-m}}\right)\quad \forall (x,t)\in\R^n\times(0,S).
$$  
Using  this  estimate and   the initial condition \eqref{eq-initial-soliton-L1},  and applying Lemma 5.1 of \cite{Hui}, we deduce  the $L^1$-contraction principle  \eqref{eq-exist-sol-contraction}  for any $0<t<T_*$. 

\medskip

\noindent{\bf Step 3.} 
Lastly, we will   prove that $T_*=T.$ Suppose to the contrary that $0<T_*<T.$
Since  $|u(\cdot,t)-U_{\lambda_0}(\cdot,t)|\in L^1(\R^n)$ for any $0<t<T_*$ by \eqref{eq-exist-sol-contraction},  direct computation shows  that  for any $0<t<T_*,$
\begin{align*}
\liminf_{R\to\infty}\frac{1}{R^{n-\frac{2}{1-m}}}\int_{|x|\leq R} u(x,t)dx
=&\liminf_{R\to\infty}\frac{1}{R^{n-\frac{2}{1-m}}}\int_{|x|\leq R} U_{\lambda_0}(x,t)dx\\
=&(T-t)^{\frac{1}{1-m}}\liminf_{R\to\infty}\frac{1 }{\left\{(T-t)^\beta R\right\}^{n-\frac{2}{1-m}}}\int_{|y|\leq (T-t)^{\beta}R} f_{\lambda_0}(y)dy.
\end{align*}
This together with \eqref{eq-est-ex-C_0} implies that for any $0<t<T_*$,
\begin{equation}\label{u-l1-loc-norm-lower-bd}
\liminf_{R\to\infty}\frac{1}{R^{n-\frac{2}{1-m}}}\int_{|x|\leq R} u(x,t)dx 
\,\ge\, c_0 (T-t)^{\frac{1}{1-m}}.
\end{equation} 
We now choose  a constant $\delta\in(0, T-T_*)$ such that  
\begin{equation*}
\delta \,<   \,\frac{1}{2}\left(\frac{c_0}{C_1}\right)^{1-m}(T-T_*-\delta),
\end{equation*}  
where we recall that  $C_1>0$ is the constant appearing   in Theorem 1.1 of \cite{Hsu2}. 
Then it follows  from  \eqref{u-l1-loc-norm-lower-bd} that 
\begin{align*}
\liminf_{R\to\infty}\frac{1}{R^{n-\frac{2}{1-m}}}\int_{|x|\leq R} u(x,T_*-\delta)dx &\geq c_0 (T-T_*+ \delta)^{\frac{1}{1-m}} > C_1(2\delta)^{\frac{1}{1-m}}.
\end{align*} 
Thus by Theorem 1.1 of \cite{Hsu2}, there exists a unique  positive solution $u_1$ to \eqref{eq-fde}  in $\R^n\times(0,2\delta)$  with initial value $u_1(x,0)=u(x,T_*-\delta)$. Since by \eqref{ab-ineqn},
\begin{equation*}
u_t(x,T_*-\delta)\leq \frac{u(x,T_*-\delta)}{(1-m)(T_*-\delta)},
\end{equation*}
  an argument similar to \cite{Hsu2} shows that  $u_1$ satisfies
\begin{equation}\label{ab-ineqn3}
u_{1,t}\le\frac{u_1}{(1-m)(t+T_*-\delta)}\quad\mbox{ in }\R^n\times(0,2\delta).
\end{equation}
  Then we extend $u$ to a function on $  \R^n\times (0, T_*+\delta)$ by letting $u(x,t)=u_1(x,t-T_*+\delta)$ in $  \R^n\times [T_*, T_*+\delta)$.  The extended function  $u$ is then a  solution of \eqref{eq-fde} in $\R^n\times (0,T_*+\delta)$ and 
 $u$ satisfies \eqref{ab-ineqn} in $\R^n\times (0,T_*+\delta)$ by \eqref{ab-ineqn3}.  
  However this contradicts the definition of the maximal existence time  $T_*$. Hence  $T_*=T$ and the theorem follows. 
\end{proof}

\section{Vanishing behavior near the  extinction time}\label{sec-extinction} 
This section is devoted to the study of  vanishing behavior  of solutions to  the fast diffusion equation \eqref{eq-fde}  near the extinction time when either \eqref{eq-cond-mono} or \eqref{eq-cond-non-mono} holds. We recall that the rescaled solution $\tilde u=\tilde u(y,\tau)$ is given  by \eqref{u-tilde-defn} and we are concerned with the large time asymptotics of $\tilde u$ as  $\tau\to\infty$. 

\subsection{Monotone increasing case \eqref{eq-cond-mono}}
\hspace{1cm}

\noindent Let
\begin{equation*}
L^1\left( \cC^{p_0};\R^n \right)=\left\{f:\int_{\R^n}|f(x)|\cC^{p_0} (x)\,dx<\infty\right\}
\end{equation*}
where $\cC(x)$ is given by \eqref{eq-def-tildeC} and 
\begin{equation}\label{p0-defn}
p_0:= \frac{1-m}{2}\left(n-\frac{2}{1-m}-\gamma_1\right).
\end{equation} 
Here $\gamma_1=\gamma_1(m,\beta)>0$ is given by  \eqref{eq-roots-gamma-intro}. Note that by \eqref{eq-exp-integ10} and \eqref{p0-defn},
\begin{equation}\label{p0-range}
0<p_0<\frac{(1-m)(n-2)}{2}\quad\mbox{ and }\quad n-\frac{2p_0}{1-m}=\frac{2}{1-m}+\gamma_1.
\end{equation}
Hence the weight function $\cC^{p_0}$ is integrable near the origin. Now we will prove  the weighted $L^1$-contraction principle for rescaled  solutions  trapped in between two self-similar profiles when the monotone increasing condition \eqref{eq-cond-mono} holds.    The following lemma can be regarded as  an extension of    Lemma 5.3 of \cite{DKS} which holds for the case $m=\frac{n-2}{n+2}$ and $n\ge 3$.  
  
\begin{prop}[weighted $L^1$-contraction principle]\label{prop-strong-contraction}
Suppose  \eqref{eq-cond-mono} holds.   Let $\tilde u_0, \tilde v_0\in L^1_{loc}(\R^n)$ satisfy 
\begin{equation*}
f_{\lambda_1}\leq\tilde u_0,\, \tilde  v_0\leq  f_{\lambda_2}\qquad\hbox{ in } \,\,\R^n
\end{equation*}
for  some constants $\lambda_2\ge \lambda_1>0$, and  
$$\tilde u_0-\tilde v_0\in L^1\left( \cC^{p_0};\R^n \right).$$ 
Let $\tilde u$ and  $\tilde v$ be the positive solutions to the rescaled   equation  \eqref{eq-fde-rescaled} in $\R^n\times (0,\infty)$ with initial values $\tilde u_0$, $ \tilde v_0 $, respectively, such that 
\begin{equation}\label{eq-tilde-sol-comparison} 
f_{\lambda_1}\leq\tilde u,\, \tilde  v\leq  f_{\lambda_2}\qquad\hbox{ in }\,\,\R^n \times(0,\infty).
\end{equation} 
Then  the following  holds. 
\begin{enumerate}[(i)]
\item For any $\tau>0 $, 
\begin{align}\label{eq-est-strong-contraction}
&\big\| ( \tilde u- \tilde v )(\cdot,\tau)\big\|_{L^1\left( \cC^{p_0};\R^n \right)}+a_*\int_0^{\tau}\int_{\R^n}\big|\left(\tilde u- \tilde v\right)  (y,s)\big|\left| \left[\frac{\cC(y)}{f_{\lambda_2}(y)}\right]^{1-m}-1 \right| \cC^{p_0}(y) \,dy\,ds\notag\\
\le&\big\| \tilde u_0- \tilde v_0   \big\|_{L^1\left( \cC^{p_0};\R^n \right)} 
\end{align}
where  
\begin{equation*}
a_*:=\frac{2m\,p_0}{C_*(1-m)}\left( \gamma_1+\frac{2m}{1-m}\right)>0.
\end{equation*}
\item  In addition,
if $  \tilde u_0-\tilde v_0  \not\equiv0$, then  
\begin{equation}\label{u-v-cancellation-ineqn}
\big\|(\tilde u-\tilde v)(\cdot,\tau)\big\|_{L^1\left( \cC^{p_0};\R^n \right)} \,< \,\big\| \tilde u_0- \tilde v_0   \big\|_{L^1\left( \cC^{p_0};\R^n \right)} \qquad\forall \tau>0.
\end{equation}
 \end{enumerate}
\end{prop}

\begin{proof} 
We will use a modification of the proof of Lemma 5.3 of \cite{DKS} to prove this proposition. 

\noindent{\bf Proof of (i)}: Let  $q:=|\tilde u-\tilde v|.$ By the Kato inequality (\cite{K} and p.89 of  \cite{DK}) and \eqref{eq-fde-rescaled},  
\begin{equation}\label{eq-q-diff-rescaled-new}
q_\tau\leq \La (\tilde aq)+\beta\Div(yq)+(\alpha-n\be) q \qquad\hbox{in $\sD'\big(\R^n \times(0,\infty)\big)$} ,
\end{equation}
where 
$$
\tilde a(y,\tau):=\int_0^1 \frac{m \,ds }{ \left(s \tilde u(y,\tau)+(1-s)\tilde v(y,\tau)\right)^{1-m}}    \qquad\forall (y,\tau)\in\R^n \times (0,\infty).
$$
By \eqref{eq-1st-asymp-ellip}, \eqref{eq-est-f-ld-tildeC} and  \eqref{eq-tilde-sol-comparison}, it holds that 
\begin{align}\label{eq-est-tilde-a} 
&\frac{m}{C_*} |y|^2 <\, m f_{\lambda_2}^{m-1}(y)  \,\leq \,\tilde a(y,\tau) \,\leq \, m f_{\lambda_1}^{m-1}(y)\leq  C \left(1+|y|^2\right) \qquad \hbox{in \,\,$\R^n \times(0,\infty)$}
\end{align} 
for some constant $C>0$. 
Let   $\eta\in C_0^\infty(\R^n)$ be a smooth cut-off function    such that $0\leq \eta\leq1, $ $\eta=1$ for $|y|\leq 1,$ and $\eta=0$ for $|y|\geq 2.$ For any $R>2$ and $0<\ve<1,$ let $\eta_{R}(y):=\eta(y/R)$, $\eta_\ve(y):=\eta(y/\ve)$, and $\eta_{\ve,R}(y)=\eta_R(y)-\eta_\ve(y).$ Then there exists a  constant $C>0$ such that $|\D \eta_{\ve,R}|^2 +|\La \eta_{\ve,R}|\leq C\ve^{-2}$   for any $\ve\le |y|\le 2\ve, $ and $|\D \eta_{\ve,R}|^2 +|\La\eta_{\ve,R}|\leq CR^{-2}$ for any $R\leq |y|\leq 2R$.   
Choosing  $\eta_{\ve,R}(y)   \cC ^{p_0}(y)$  as  a test function   in   \eqref{eq-q-diff-rescaled-new}   yields that  for any  $\tau>0,$ 
\begin{align}\label{q-ineqn}
&\int_{\R^n} q (y,\tau)\eta_{\ve,R}(y)\cC^{p_0}(y)  dy-\int_{\R^n} q (y,0)\eta_{\ve,R}(y)\cC^{p_0} (y) dy\notag\\
\le&\int_0^\tau\int_{\R^n} \Big\{\tilde a\La   \cC^{p_0}(y)  -\beta y\cdot\D   \cC^{p_0}(y) +(\alpha-n\be)   \cC^{p_0}(y) \Big\} q\, \eta_{\ve,R}  \,dy\,ds\notag\\
&\qquad  +   \int_0^\tau\int_{B_{2R}\setminus B_R}  \Big\{\tilde a\La \eta_{\ve,R}    \cC^{p_0}(y) +2 \tilde a \D \eta_{\ve,R}\cdot\D   \cC^{p_0}(y)  -\be y\cdot\D \eta_{\ve,R}   \cC^{p_0}(y) \Big\}q\,dy\,ds\notag\\
&\qquad+\int_0^\tau \int_{B_{2\ve}\setminus B_\ve}  \Big\{\tilde a\La \eta_{\ve,R}    \cC^{p_0}(y) 
+2 \tilde a \D \eta_{\ve,R}\cdot\D   \cC^{p_0}(y)  -\be y\cdot\D \eta_{\ve,R}    \cC^{p_0}(y)  \Big\}q\,dy\,ds.
\end{align}
Since $\gamma_1=\gamma_1(m,\beta)$ is a root of the characteristic equation \eqref{eq-char}, by  \eqref{p0-defn}, \eqref{p0-range} and a direct computation, we have
\begin{equation}\label{eq-q-diff-la-C}
 \frac{m }{C_*}|y|^{2}\cdot\La\cC^{p_0}(y)=-\frac{2mp_0}{C_*(1-m)}\left(n-2- \frac{2p_0}{1-m}\right)\cC^{p_0}(y)=-a_*\cC^{p_0}(y) <0\quad\forall y\in\R^n\setminus\{0\},
\end{equation} 
and
\begin{equation}\label{eq-q-diff-rescaled-laplace-test-ft-1}
\begin{aligned}
 &\frac{m}{C_*}|y|^{2} \cdot \La   \cC^{p_0}(y)  -\beta y\cdot\D   \cC^{p_0}(y) +(\alpha-n\be)   \cC^{p_0}(y)\\
 &=   \left\{\frac{m}{C_*}\left(\gamma_1-n+\frac{2}{1-m}\right)\left(\gamma_1+\frac{2m}{1-m}\right)  +\beta\left(n-\frac{2}{1-m}-\gamma_1\right) +\alpha-n\beta \right\}\cC^{p_0}(y)\\
&= \frac{m}{C_*} \left\{ \gamma_1^2- \frac{A_0(m,\beta)}{1-m}\, \gamma_1+\frac{2(n-2-nm)}{1-m} \right\}\cC^{p_0}(y)=0  \qquad\forall y\in\R^n\setminus\{0\}. 
\end{aligned}
\end{equation} 
On the other hand by  \eqref{eq-tilde-sol-comparison} and  Theorem \ref{thm-self-sol-2nd-asymp}, there exists a constant $C>0$ such that 
\begin{equation}\label{eq-tilde-sol-q}
q(y,\tau)\leq|f_{\lambda_1}-f_{\lambda_2}|(y)\leq   C \min\left(1,  |y|^{-\frac{2}{1-m}-\gamma_1}\right)\qquad   \forall (y,\tau)\in \R^n \times(0,\infty).  
\end{equation} 
By \eqref{eq-est-tilde-a}, \eqref{q-ineqn}, \eqref{eq-q-diff-rescaled-laplace-test-ft-1}, \eqref{eq-tilde-sol-q}, and the properties of $\eta_{\ve,R}$, it follows that for any $0<\varepsilon<1$,   $R>2$,  and $\tau>0$, 
\begin{align}\label{eq-q-diff-rescaled-integration-1}
&\int_{\R^n}q(y,\tau)\eta_{\ve,R}(y)\cC^{p_0}(y)\,dy+\int_0^\tau\int_{\R^n}\left(\frac{m}{C_*}|y|^{2}-\tilde a(y,s) \right)\La   \cC^{p_0}(y)q(y,s)\eta_{\ve,R}(y)\,dy\,ds\notag\\
\le&\int_{\R^n}q(y,0)\cC^{p_0} (y)\, dy+C\int_0^\tau\int_{B_{2R}\setminus B_R}q(y,s)\cC^{p_0}(y)\,dy\,ds+C\int_0^{\tau}\int_{B_{2\varepsilon}\setminus B_{\varepsilon}}(1+\varepsilon^{-2})\varepsilon^{-\frac{2p_0}{1-m}}\,dy\,ds \notag\\
\le&\int_{\R^n} q (y,0)  \cC^{p_0} (y)\, dy +C\int_0^\tau\int_{B_{2R}\setminus B_R}q(y,s)\cC^{p_0}(y)\,dy\,ds+C\ve^{n-2 -\frac{2p_0}{1-m}}\,\tau
\end{align}  
for some generic constant $C>0$ which may vary from line to line. By \eqref{eq-est-f-ld-tildeC}, \eqref{eq-est-tilde-a} and
 \eqref{eq-q-diff-la-C}, 
\begin{align}\label{a-tilde-c-relation2}
\left(\frac{m}{C_*}|y|^{2}-\tilde a(y,\tau)\right)\La\cC^{p_0}(y)
&=a_*\cC^{p_0}(y)\left(\,\frac{1}{m}\tilde a(y,\tau)\cC(y)^{1-m}-1 \right)\notag\\
&\ge a_*\cC^{p_0}(y) \left(\left[\frac{  \cC(y)}{f_{\lambda_2}(y)}\right]^{1-m}-1 \right)>0\quad\forall y\in \R^n,\tau>0.
\end{align}
Since   \eqref{p0-range} implies  
\begin{equation*}
n-2-\frac{2p_0}{1-m}\, =\,\gamma_1+\frac{2m}{1-m}>0,
\end{equation*}
 letting $\ve\to 0$ in \eqref{eq-q-diff-rescaled-integration-1}, we deduce  by \eqref{a-tilde-c-relation2}  that 
\begin{align}\label{eq-q-diff-rescaled-integration-2}
&\int_{\R^n}q(y,\tau)\eta_R(y)\cC^{p_0}(y)\,  dy+a_* \int_0^\tau\int_{\R^n}   q (y,s) \eta_R(y)  \cC^{p_0}(y)  \left| \left[\frac{  \cC(y)}{f_{\lambda_2}(y)}\right]^{1-m}-1 \right|\,dy\,ds\notag\\
\le&\int_{\R^n} q (y,0)\cC^{p_0} (y)\, dy + C\int_0^\tau\int_{B_{2R}\setminus B_R}q(y,s)\cC^{p_0}(y)\,dy\,ds\qquad\forall R>2,\,\, \tau>0.
\end{align}
This together with  \eqref{p0-range} and \eqref{eq-tilde-sol-q}  yields  that 
\begin{align}\label{eq-q-diff-rescaled-integration-10} 
&\int_{\R^n} q (y,\tau)\eta_R(y)\cC^{p_0}(y)\,  dy+a_* \int_0^\tau\int_{\R^n}q(y,s)\eta_R(y)\cC^{p_0}(y)  \left| \left[\frac{\cC(y)}{f_{\lambda_2}(y)}\right]^{1-m}-1 \right|\,dy\,ds\notag\\
\le&\int_{\R^n} q (y,0)\cC^{p_0} (y)\, dy +CR^{n-\frac{2}{1-m} -\gamma_1 -\frac{2p_0}{1-m} }\,\tau\notag\\
\le&\int_{\R^n} q (y,0)\cC^{p_0} (y)\, dy +C\tau 
\end{align} 
for some constant $C>0$. Letting $R\to\infty$ in \eqref{eq-q-diff-rescaled-integration-10}, it follows that 
\begin{align*}
&\int_{\R^n} q (y,\tau)\cC^{p_0}(y)\,  dy+a_* \int_0^\tau\int_{\R^n}   q (y,s) \cC^{p_0}(y)  \left| \left[\frac{\cC(y)}{f_{\lambda_2}(y)}\right]^{1-m}-1 \right|\,dy\,ds\notag\\
&\le \int_{\R^n} q (y,0)  \,\cC^{p_0} (y) dy + C\tau.
\end{align*} 
Hence 
\begin{equation}\label{eq-q-diff-rescaled-integration-xt-new-2}
\int_0^\tau\int_{\R^n} q (y,s)\cC^{p_0}(y) \, dy\,ds\leq \tau\int_{\R^n} q (y,0)\cC^{p_0} (y) dy+C\tau^2\qquad\forall \tau>0. 
\end{equation} 
Letting $R\to\infty$ in \eqref{eq-q-diff-rescaled-integration-2} with the use of  \eqref{eq-q-diff-rescaled-integration-xt-new-2}, we get    \eqref{eq-est-strong-contraction}. 

\medskip

\noindent{\bf Proof of (ii)}: Now we also assume that  $q(\cdot,0)=|\tilde u_0-\tilde v_0|\not\equiv0$. Then  
$$
 \int_0^\tau\int_{\R^n}   q(y,s)  \cC^{p_0}(y)  \left| \left[\frac{\cC(y)}{f_{\lambda_2}(y)}\right]^{1-m}-1\right|\,dy\,ds>0\qquad\forall\tau>0.
$$
This  together with  \eqref{eq-est-strong-contraction}  implies  \eqref{u-v-cancellation-ineqn}, and the proposition follows.
\end{proof}

By using  \eqref{p0-range} and Theorem \ref{thm-self-sol-2nd-asymp}, we have the following lemma.

\begin{lemma}
Suppose \eqref{eq-cond-mono} holds. Let $\lambda_2>\lambda_1>0$ and $p_0$ be given by \eqref{p0-defn}.   Then $f_{\lambda_1}-f_{\lambda_2}\not \in L^1( \cC^{p_0};\R^n)$.
\end{lemma}

Based on the strong    $L^1$-contraction     principle in Proposition \ref{prop-strong-contraction},   we prove  the convergence   of the rescaled solution $\tilde u(\cdot, \tau)$ to a self-similar profile as $\tau\to\infty$  in Theorem \ref{thm-uniform-convergence-tilde-u-compact} in the monotone case. 
  
\begin{proof}[\textbf{Proof of Theorem \ref{thm-uniform-convergence-tilde-u-compact}}]   
Let $\{\tau_i\}_{i=1}^\infty$ be a sequence such that $\tau_i\to\infty$  as $i\to\infty$  and  let 
\begin{equation}\label{eq-def-scaled-time-shift}
\tilde u_i(y, \tau):=\tilde u(y, \tau_i+\tau)\qquad\forall (y,\tau)\in\R^n\times[-\tau_i,\infty).
\end{equation}  
By \eqref{u-lower-upper-bd-ini}  and  Theorem \ref{thm-existence-fde},  we have 
\begin{align}\label{tilde-u-lower-upper-bd} 
&U_{\lambda_1}\, \leq \,  u\, \leq \, U_{\lambda_2} \quad\hbox{ in } \,\,\R^n\times(0,T)\notag\\
\Rightarrow\quad&f_{\lambda_1}\le\tilde u_i\leq  f_{\lambda_2}\qquad\hbox{ in }\,\,\R^n \times(-\tau_i,\infty),\quad\forall i\in\Z^+.
\end{align}
For any $N>0$, we  choose  $i_{N}\in \Z^+$ such that $\tau_i>N$ for any $i\ge i_{N}$. Then by \eqref{tilde-u-lower-upper-bd},  the equation \eqref{eq-fde-rescaled} for $\{\tilde u_i\}_{i=i_{N}}^{\infty}$ is uniformly parabolic in $B_R\times(-N, \infty) $ for any $R>1$. 
Thus  utilizing  the parabolic Schauder estimates \cite{LSU}, the sequence $\{ \tilde u_i\}_{i=i_N}^\infty$ is equi-H\"older continuous  in $C^{2,1}(K)$ for any compact subset $K\subset \R^n \times(-N, \infty).$
By the Ascoli Theorem and a diagonalization argument, there exist  a subsequence of the sequence $\{\tilde u_i\}_{ i=1}^\infty$, which we still denote by $\{\tilde u_i\}_{i=1}^\infty$, and a  function $\tilde v \in C^{2,1}(\R^n\times\R)$   such that $\tilde u_i$ converges to  $\tilde v$   in $C^{2,1}(K)$ as $i\to\infty$ for any compact set $K\subset\R^n\times\R$. 
Then the limit $\tilde v$ is an eternal solution to \eqref{eq-fde-rescaled} in $\R^n\times\R$  and  satisfies   
\begin{equation*}
f_{\lambda_1}   \leq   \tilde v \leq  f_{\lambda_2} \quad\mbox{ in   {  $\R^n \times\R$}   }
\end{equation*}   
in view of \eqref{tilde-u-lower-upper-bd}. Let $p_0$ be given by \eqref{p0-defn}.
By \eqref{eq-initial-soliton-L1}, \eqref{u-lower-upper-bd-ini} and the fact that $\cC^{p_0}$ is integrable near the origin, it holds that 
\begin{align*}
&u_0-U_{\lambda_0}(\cdot,0)\in L^1\left( \cC^{p_0};\R^n \right)\\ 
\Rightarrow\quad&\tilde u_0 - f_{\lambda_0}\in L^1\left( \cC^{p_0};\R^n \right).
\end{align*} 
Hence by Proposition \ref{prop-strong-contraction},  we have 
\begin{equation}\label{ineqn1}
a_*\int_0^{\infty}\int_{\R^n}\left|\left(\tilde u- f_{\lambda_0}\right)(y,s)\right|\left|\left[\frac{\cC(y)}{f_{\lambda_2}(y)}\right]^{1-m}-1\right|\cC^{p_0}(y)\,dy\,ds
\le  \big\| \tilde u_0- f_{\lambda_0}\big\|_{L^1\left( \cC^{p_0};\R^n \right)}<\infty. 
\end{equation}
Since    for any  $N>0$ and $i\ge i_{N}$,
\begin{align}\label{ineqn2}
&\int_{-N}^{\infty}\int_{\R^n}\big|\left(\tilde u_i- f_{\lambda_0}\right)  (y,s)\big|\left| \left[\frac{\cC(y)}{f_{\lambda_2}(y)}\right]^{1-m}-1 \right|\cC^{p_0}(y)  \,dy\,ds\notag\\
=&\int_{\tau_i-N}^{\infty}\int_{\R^n}\left|\left(\tilde u- f_{\lambda_0}\right)  (y,s)\right|\left| \left[\frac{\cC(y)}{f_{\lambda_2}(y)}\right]^{1-m}-1 \right|\cC^{p_0}(y)  \,dy\,ds,
\end{align}
letting $i\to \infty$ in \eqref{ineqn2},  we deduce from  \eqref{ineqn1} and the Fatou lemma that  
\begin{align*}
&\int_{-N}^{\infty}\int_{\R^n}\big|\left(\tilde v- f_{\lambda_0}\right)  (y,s)\big|\left| \left[\frac{\cC(y)}{f_{\lambda_2}(y)}\right]^{1-m}-1 \right|\cC^{p_0}(y)\,dy\,ds=0 \qquad\forall N>0. 
\end{align*} 
This yields that 
\begin{equation*}
\tilde v\equiv f_{\lambda_0}\quad\mbox{ in }\R^n\times\R.  
\end{equation*}
Hence we conclude that  $\tilde u(\cdot, \tau_i)=\tilde u_i(\cdot, 0)$ converges to $f_{\lambda_0}$ uniformly in $C^2(K)$ as $i\to\infty$ for any  compact subset $K\subset\R^n$. Since the sequence $\{\tau_i\}_{i=1}^{\infty}$ is arbitrary,  $\tilde u (\cdot,\tau)$ converges to $f_{\lambda_0}$ uniformly in $C^2(K)$ as $\tau\to\infty$ for any  compact subset $K\subset\R^n$. 

Lastly we  assume that (i) of \eqref{eq-cond-mono} holds. Note that
\begin{align}\label{u-tilde-u-l1=}
\int_{\R^n} |\tilde u(y,\tau)-f_{\lambda_0}(y)|dy=& (T-t)^{n\beta-\alpha}\int_{\R^n} |u(x,t)-U_{\lambda_0}(x,t)|dx\notag\\
=&T^{n\beta-\alpha}e^{-(n\beta-\alpha)\tau}\int_{\R^n} |u(x,t)-U_{\lambda_0}(x,t)|dx
\end{align}
where $\tau=-\log \{(T-t)/ T\}>0$. Then by \eqref{eq-exist-sol-contraction} and \eqref{u-tilde-u-l1=}, we get \eqref{eq-resc-sol-exp} and the theorem follows.  
\end{proof}

\subsection{Non-monotone case \eqref{eq-cond-non-mono}}

In this subsection, we will prove Theorem \ref{thm-uniform-convergence-tilde-u-0}  regarding the  asymptotic behavior near the extinction time in the non-monotone case of \eqref{eq-cond-non-mono}.     In this case we employ a different approach from  the  monotone case of \eqref{eq-cond-mono}  in order to provide a convergence result of  the rescaled solution $ \tilde u(\cdot,\tau)$  to zero as $\tau\to\infty$. 
 
\begin{proof}[\textbf{Proof of Theorem \ref{thm-uniform-convergence-tilde-u-0}}]
Let $\{\tau_i\}_{i=1}^\infty$ be a sequence such that $\tau_i\to\infty$  as $i\to\infty$  and    
let $\tilde u_i $ be given by \eqref{eq-def-scaled-time-shift}.  By \eqref{eq-tilde-sol-upper-bd-ini} and Theorem \ref{thm-existence-fde},  $u$ satisfies 
\begin{align}
&0\le u\le\min\big(U_{\lambda_1},\,U_{\lambda_2} \big)\quad\,\,\hbox{in }\, \,\R^n\times(0,T)\notag\\ 
\Rightarrow\quad&0\le\tilde u\leq\min\big(f_{\lambda_1},f_{\lambda_2}\big)\qquad\,\hbox{in }\,\,\R^n\times (0,\infty)\label{eq-tilde-sol-upper-bd}\\
\Rightarrow\quad&0\le\tilde u_i\leq\min\big(f_{\lambda_1},f_{\lambda_2}\big)\qquad\hbox{in } \,\,\R^n\times(-\tau_i,\infty),\quad i\in\Z^+.
\label{eq-tilde-sol-i-upper-bd}
\end{align}  
For any $N>0$, we  choose  $i_{N}\in \Z^+$ such that $\tau_i>N$ for any $i\ge i_{N}$.  Then by \eqref{eq-tilde-sol-i-upper-bd} and   Theorem 1.1 of \cite{Sa},   the sequence $\{\tilde u_i\}_{i=i_{N}}^{\infty}$ is equi-H\"older continuous on any  compact subset  $K$ of $\R^n\times (-N,\infty)$. 
Hence by the Ascoli Theorem and a diagonalization argument, there exist a subsequence of the sequence $\{\tilde u_i\}_{ i=1}^\infty$,   still denoted by $\{\tilde u_i\}_{i=1}^\infty$,  and a function  $\tilde u_{\infty} \in C(\R^n\times\R)$ such that $\tilde u_i$ converges to   $\tilde u_{\infty}  $ uniformly in $K$ as $i\to\infty$ for any compact subset $K\subset\R^n\times\R$.
Then the limit function $\tilde u_{\infty}$ solves     \eqref{eq-fde-rescaled} in $\sD'\big(\R^n \times\R\big)$. 
Letting $i\to\infty$ in  \eqref{eq-tilde-sol-i-upper-bd},   
\begin{equation} \label{tilde-u-infty-upper-bd0}
0\leq    \tilde u_{\infty} \leq\min\big(f_{\lambda_1},f_{\lambda_2}\big)\qquad\hbox{ in } \,\,\R^n\times\R.
\end{equation} 

Now we will show that $\tilde u_\infty \equiv 0$ on $\R^n\times\R$.   In order to prove this, we let
\begin{equation}\label{eq-def-ld-tau}
\lambda(\tau):= \inf \big\{ \lambda>0\,:\,   \tilde u(y,\tau)\le f_{\lambda}(y)\quad \forall y\in\R^n\big\}\qquad\forall\tau> 0.
\end{equation}
By \eqref{eq-tilde-sol-upper-bd} and \eqref{eq-def-ld-tau},  $\lambda (\tau)$ is well-defined,
\begin{align}
&\tilde u(y,\tau)\le f_{\lambda(\tau)}(y)\qquad\qquad\qquad\qquad\,\,\forall y\in\R^n,\,\,\tau>0\label{u-tilde-f-compare5} \\
\Rightarrow\quad&0< \tilde u(0,\tau)\leq f_{\lambda(\tau)}(0)=\lambda(\tau)^{\frac{2}{1-m}}\qquad\forall \tau>0,\notag
\end{align}
and
\begin{equation}\label{eq-est-upper-bd-ld=tau}
0<\lambda (\tau)\le\lambda_1\qquad\forall\tau>0.
\end{equation}
By \eqref{eq-tilde-sol-upper-bd} and  Lemma \ref{lem-rev-monotonicity},  there exists a constant $R_*>0$  such that
\begin{equation}\label{eq-est-behavior-at-infty}
0< \tilde u(y,\tau)\leq f_{\lambda_2}(y)<f_{\lambda}(y)\quad\forall 0<\lambda\leq\lambda_1,\,\,|y|\geq R_*/\lambda,\,\, \tau>0.
\end{equation} 
  
\begin{Claim}\label{cla-2-touching} For each $\tau>0,$ $f_{\lambda(\tau)}$ touches $\tilde u(\cdot,\tau)$ from above in $B_{2R_*/\lambda(\tau)}$. 
\end{Claim} 
\noindent{\bf Proof of Claim \ref{cla-2-touching}}: Suppose there exists $\tau_1>0$ such that the claim does not hold. Then by \eqref{u-tilde-f-compare5}, 
\begin{equation}\label{u-tidle-f-ineqn}
 \tilde u(y,\tau_1)< f_{\lambda(\tau_1)}(y)\qquad\hbox{ in }B_{ {2R_*/\lambda(\tau_1)}}. 
\end{equation}
By  \eqref{u-tidle-f-ineqn} and the continuity of $\tilde u(\cdot,\tau_1)$ and $f_{\lambda}$, there exists 
a  constant $0<\delta<\lambda(\tau_1)$ such that  
$$
\tilde u(\cdot,\tau_1)< f_{\lambda}\quad \hbox{ in }\overline{B_{ {R_*/\lambda}}}\qquad\forall\lambda(\tau_1)-\delta<\lambda<\lambda(\tau_1). 
$$
 This together with   \eqref{eq-est-upper-bd-ld=tau} and  \eqref{eq-est-behavior-at-infty}  yields that 
\begin{equation*}
 \tilde u(\cdot,\tau_1)< f_{\lambda}\quad \hbox{ in }\R^n\qquad\forall\lambda(\tau_1)-\delta<\lambda<\lambda(\tau_1).
\end{equation*} 
This contradicts the definition of $\lambda(\tau_1)$ in \eqref{eq-def-ld-tau}. Thus no such $\tau_1$ exists and Claim \ref{cla-2-touching} follows. 
\medskip 
  
\begin{Claim}\label{cla-2-mono}
$\lambda(\tau)$ is a strictly decreasing function of $\tau>0$. 
\end{Claim} 
\noindent{\bf Proof of Claim \ref{cla-2-mono}}:
We fix $ \tau_1>0$ and let $t_1:=T(1-e^{-\tau_1})$. By \eqref{u-tilde-f-compare5} with $\tau=\tau_1$,   we have
\begin{equation}\label{u<U-lambda10}
u(x,t_1)\le U_{\lambda(\tau_1)}(x,t_1)\qquad\forall x\in\R^n.
\end{equation}
Then arguing similarly as for \eqref{u-u-lambda-ineqn2} with the use of  \eqref{u<U-lambda10}, we have
\begin{align}
&u(x,t)\le U_{\lambda(\tau_1)}(x,t)\qquad\forall x\in\R^n,\,\, t_1<t<T\label{eq-compare-u-U}\\
\Rightarrow\quad&\tilde{u}(y,\tau)\le f_{\lambda(\tau_1)}(y)\quad\qquad\forall y\in\R^n,\,\, \tau>\tau_1\notag\\
\Rightarrow\quad&\lambda(\tau)\le\lambda(\tau_1)\qquad\forall \tau>\tau_1.\notag
\end{align}
Hence $\lambda(\tau)$ is a decreasing function of $\tau>0$.

By \eqref{eq-est-upper-bd-ld=tau} and \eqref{eq-est-behavior-at-infty} with $\tau=\tau_1$, 
\begin{equation}\label{u--u-lambda-compare}
u(x,t_1)< U_{\lambda(\tau_1)}(x,t_1)\qquad  \forall |x|\geq R_*/[(Te^{-\tau_1})^{\beta}\lambda(\tau_1)] .
\end{equation} 
Then by the continuity of $u$ and $ U_{\lambda(\tau_1)}$ and \eqref{u--u-lambda-compare}, there exist constants $R_1> R_*/[(Te^{-\tau_1})^{\beta}\lambda(\tau_1)]$ and $t_2\in (t_1,T)$   such that  
\begin{equation}\label{u--u-lambda-compare2}
u(x,t)< U_{\lambda(\tau_1)}(x,t)\qquad \forall |x|=R_1,\,\, t_1\le t\le t_2. 
\end{equation}  
Since by \eqref{eq-1st-asymp-ellip},
$$
U_{\lambda(\tau_1)}(x,t)\geq  c_0\min\left(1, |x|^{-\frac{2}{1-m}}\right)\qquad \forall (x,t)\in\R^n\times [t_1,t_2]
$$   
for some constant $c_0 >0$, using {\eqref{eq-compare-u-U}, \eqref{u--u-lambda-compare},  \eqref{u--u-lambda-compare2},}  and the strong comparison principle   for an exterior   domain  in   Lemma \ref{lem-strict-comp-prin} implies that
\begin{equation}\label{u--u-lambda-compare3}
u(x,t)< U_{\lambda(\tau_1)}(x,t)\qquad \forall |x|\ge R_1,\,\, t_1< t<t_2.
\end{equation} 
On the other hand, by \eqref{u<U-lambda10}, \eqref{u--u-lambda-compare2} and the strong comparison principle  for a bounded domain in Lemma \ref{lem-strict-comp-prin2}, it follows that
\begin{equation}\label{u--u-lambda-compare5}
u(x,t)< U_{\lambda(\tau_1)}(x,t)\qquad \forall |x|< R_1,\,\, t_1<t<t_2.
\end{equation} 
Let  $\tau_2=-\log \{(T-t_2)/T\}$.  Then by  \eqref{u--u-lambda-compare3} and \eqref{u--u-lambda-compare5},
\begin{equation}\label{u-tilde-u-lambda-compare6}
\tilde{u}(y,\tau)< f_{\lambda(\tau_1)}(y)\qquad \forall y\in\R^n,\,\, \tau_1<\tau<\tau_2.
\end{equation} 
Since  $\lambda (\tau)\leq \lambda (\tau_1)$ for any $\tau> \tau_1$, 
employing  \eqref{u-tilde-u-lambda-compare6} and  Claim \ref{cla-2-touching} yields  that  $\lambda (\tau)<\lambda (\tau_1)$ for any $\tau\in(\tau_1,\tau_2)$. Hence $\lambda(\tau)$ is a strictly decreasing function of $\tau>0$ finishing the proof of Claim \ref{cla-2-mono}.
  \medskip

\noindent By \eqref{eq-est-upper-bd-ld=tau} and Claim \ref{cla-2-mono},    the limit 
\begin{equation*}
\lambda_\infty:=\lim_{\tau\to\infty}\lambda(\tau)\,\in\, [0,\lambda_1)
\end{equation*}
exists. Then by  \eqref{u-tilde-f-compare5},  we have 
\begin{equation}\label{eq-est-tilde-v-f-infty}  
\tilde u_{\infty}(y,\tau )= \lim_{i\to\infty}\tilde u(y,\tau_i+\tau)\leq  \lim_{i\to\infty} f_{\lambda(\tau_i+\tau)}(y)=f_{\lambda_\infty}(y)\qquad\forall (y,\tau)\in\R^n\times\R.
\end{equation}
Moreover, it follows from 
  \eqref{tilde-u-infty-upper-bd0} and \eqref{eq-est-tilde-v-f-infty} that 
\begin{equation} \label{eq-tilde-sol-upper-bd-ini-v}
 0\leq  \tilde u_{\infty}  \leq\min\big(f_{\lambda_\infty},f_{\lambda_1}\big)\quad\mbox{ in } \,\,\R^n\times \R.
\end{equation}

\begin{Claim}\label{u-tilde-infty=0}
$\tilde u_{\infty}\equiv0\,$ in $\R^n\times\R.$
\end{Claim} 

Note that once we have proved Claim \ref{u-tilde-infty=0},      $\tilde  u(\cdot,\tau_i)$ converges to zero uniformly on any compact subset of $\R^n$ as $i\to\infty$.  Since the sequence $\{\tau_i\}^\infty_{i=1}$ is arbitrary,  we can then conclude that $\tilde u(\cdot,\tau)$ converges to zero uniformly on any compact subset of  $\R^n$ as $\tau\to\infty$, completing the proof of Theorem \ref{thm-uniform-convergence-tilde-u-0}.  So it remains  to prove  Claim \ref{u-tilde-infty=0}. 

\medskip
 
\noindent{\bf Proof of Claim \ref{u-tilde-infty=0}}: 
Suppose to the contrary that  Claim \ref{u-tilde-infty=0} does not hold. Then without loss of generality we may assume  that $\tilde u_{\infty}(0,0)>0$.   By the continuity of $\tilde u_{\infty}$,  there exists  a constant $\delta>0$ such that
\begin{equation}\label{u-tilde-loc-positive}
\tilde u_{\infty}(y,\tau)>0\qquad\forall |y|<\delta,\,\, |\tau|<\delta.
\end{equation}
We now define  
\begin{equation}\label{w-defn}
w(x, t):=(T-t)^{\alpha}\,\tilde u_\infty\left((T-t)^{\beta}x,\tau\right) \qquad\hbox{in \,\,$\R^n\times(-\infty,T)$}
\end{equation}
 with $\tau=-\log \{(T-t)/T \}$. 
Since $\tilde u_{\infty}\in C(\R^n\times\R)$ solves \eqref{eq-fde-rescaled} in $\sD'(\R^n\times\R)$, the nonnegative function $ w\in C\left(\R^n\times(-\infty,T)\right)$ satisfies \eqref{fast-diff-eqn} in $\sD'\left(\R^n\times(-\infty,T)\right)$. Thus by \eqref{u-tilde-loc-positive} and \eqref{w-defn} we can apply    \cite[Lemma 3.3]{HuiK} to  $w$ in order to conclude that
\begin{equation}\label{u-tilde-infty-positive10}
\tilde u_{\infty} >0\qquad \hbox{in \,\, $  \R^n\times (-\delta,\delta) .$}
\end{equation} 
By \eqref{eq-tilde-sol-upper-bd-ini-v}, it follows that 
\begin{equation*}
0< \tilde u_{\infty}(0,0)\leq f_{\lambda_{\infty}}(0)=\lambda_{\infty}^{\frac{2}{1-m}}\quad\Rightarrow\quad\lambda_{\infty}>0.
\end{equation*}

Let $K$ be any  compact subset of $\R^n\times (-\delta,\delta)$ and  $\displaystyle \nu_K=\min_K\tilde u_{\infty}$. Then by \eqref{u-tilde-infty-positive10},  $\nu_K>0$. 
Since $\tilde u_i$ converges to $\tilde u_{\infty}$  uniformly on $K$ as $i\to\infty$, there exists $i_K\in\Z^+$ such that
\begin{equation}\label{tilde-ui-uniform-lower-bd}
 \min_K \tilde u_i >\nu_K/2>0\qquad\forall i\geq i_K.
\end{equation}
By \eqref{eq-tilde-sol-i-upper-bd} and \eqref{tilde-ui-uniform-lower-bd}, the equation \eqref{eq-fde-rescaled} for the sequence $\{\tilde u_i\}_{i\ge i_K}$ is uniformly parabolic on $K$. Then
by the   parabolic Schauder estimates \cite{LSU}, the sequence $\{\tilde u_i\}_{i\ge i_K}$ is equi-H\"older continuous  in $C^{2,1}(K_1)$ for any compact subset $K_1\Subset K$. Since      compact subsets $K_1\Subset K$ of $\R^n\times (-\delta,\delta)$ are arbitrary,  by the Ascoli Theorem and a diagonalization argument,  there exists  a subsequence of the sequence $\{\tilde u_i\}_{ i=1}^\infty$, still denoted by $\{\tilde u_i\}_{i=1}^\infty$ such that $\tilde u_i$ converges to $\tilde u_{\infty}$ uniformly in $C^{2,1}(K)$ as $i\to\infty$ for any compact set   $K$ of $\R^n\times (-\delta,\delta)$. Hence, $\tilde u_{\infty}\in C^{2,1}(\R^n\times (-\delta,\delta))$ is a classical positive solution to \eqref{eq-fde-rescaled}  in $\R^n\times (-\delta,\delta)$. 
  
Now we define 
$$
\mu(\tau):= \inf \big\{\, \mu>0:   \tilde u_{\infty}(y,\tau)\le f_{\mu}(y)\quad \forall y\in\R^n\big\}\qquad\forall\tau\in (-\delta,\delta).
$$ 
By \eqref{eq-tilde-sol-upper-bd-ini-v}, \eqref{u-tilde-infty-positive10} and using the same argument as before, we  deduce that 
\begin{equation}\label{eq-mu-ld1}
0< \mu(\tau')<\mu(\tau)<\mu(0)\leq \lambda_\infty<\lambda_1\qquad\forall 0< \tau<\tau'<\delta.
\end{equation}
We fix $s_1\in(0,\delta)$.   Arguing similarly as for  \eqref{u-tilde-u-lambda-compare6},  there exists a constant $s_2\in (s_1,\delta)$ such that 
 \begin{equation} \label{eq-tilde-sol-v-contra}
\tilde u_{\infty}(\cdot,s_2)<f_{\mu(s_1)}\qquad\hbox{in\,\, $  \R^n$}.
\end{equation}
 Let    $R_*>0$ be the constant appearing in \eqref{eq-est-behavior-at-infty}.  
Since  $\tilde u_i(\cdot,  s_2)$ converges   to  $\tilde u_{\infty}(\cdot, s_2)$  uniformly  on $\overline{B_{R_*/\mu(s_1)}}$ as $i\to\infty$, by  \eqref{eq-tilde-sol-v-contra} there exists  $i_0\in\Z^+$ such that
$$
\tilde u\left(y,\tau_i+s_2\right)=\tilde u_i(y,s_2)<f_{\mu(s_1)}(y)\qquad\forall y\in\overline{B_{R_*/ \mu(s_1)}},\,\, i\ge i_0.
$$  
This together with \eqref{eq-est-behavior-at-infty} and \eqref{eq-mu-ld1} implies that  
\begin{equation}\label{tilde-ui<f-lambda2}
\tilde  u(y,\tau_i+ s_2)<f_{\mu(s_1)}(y)\qquad \forall y\in\R^n,\,\,i\ge i_0.
\end{equation}
By Claim \ref{cla-2-mono} and \eqref{tilde-ui<f-lambda2} we deduce that  
 $$
 \lambda_\infty < \lambda(\tau_i+s_2)\leq   \mu(s_1)\qquad \forall i\ge i_0 ,
 $$ 
which  contradicts \eqref{eq-mu-ld1}, and therefore Claim \ref{u-tilde-infty=0} follows. This  finishes the proof of Theorem \ref{thm-uniform-convergence-tilde-u-0}.
\end{proof}


\section*{Acknowledgment}
 The second  author  is   supported by National Research Foundation of Korea (NRF) Grant No. NRF-2018R1C1B6003051.  


\appendix
\section{Strong comparison principle}
In this section, we will establish   the comparison principle for the fast diffusion equation  in   exterior domains and in    bounded domains. Firstly, the following lemma  deals with 
   the comparison principle   on   exterior domains in the range $0<m<\frac{n-2}{n}$, $n\ge 3$. We refer to Lemma 3.4  of \cite{HP} for the supercritical case when $\frac{(n-2)_+}{n}<m<1$.

\begin{lemma}[Strong comparison principle in exterior domains]
\label{lem-strict-comp-prin}
Let $n\ge 3$, $0<m<\frac{n-2}{n}$, $ R_1>0$ and $T_1>0$.  Let  $u_1$,   $u_2\in  C^{2,1}\big((\R^n\setminus\overline{B_{R_1}})\times (0,T_1]\big)\cap C\big((\R^n\setminus B_{R_1})\times [0,T_1]\big)$  be    positive   solutions  to
$$u_t=\La u^m \qq\hbox{in \,\,$(\R^n\setminus\overline{B_{R_1}})\times (0,T_1]$  } $$ 
 such that   \begin{equation} \label{eq-comp-bdry}
\left\{\,
\begin{aligned}
  u_{1}(x,0)\le u_{2}(x,0) \qquad&\forall     |x| \geq  R_1  ;\\
u_1(x,t)<u_2(x,t)\qquad&\forall |x|=R_1,\,\, 0<t\leq  T_1,
\end{aligned}\right.
\end{equation}
 and 
\begin{equation}\label{eq-est-lower-bd}
\min(u_1(x,t),u_2(x,t))\geq\, c_0\, |x|^{-\frac{2}{1-m}}\qquad\forall |x|\geq R_1,\,\, 0\leq t\leq T_1
\end{equation}
  for some constant $c_0>0$. In addition, we assume  that 
\begin{equation}\label{u1-u2-l1-norm-bd}
(u_1-u_2)_+\in L^{1}\left((0,T_1);L^1(\R^n\setminus \overline{B_{R_1}})\right).
\end{equation}
Then 
\begin{equation}\label{u1<u2-exterior-domain}
u_1(x,t)<u_2(x,t)\qquad\forall |x|>R_1,\,\, 0<t\leq  T_1.
\end{equation}
\end{lemma}
\begin{proof}
We will first  show that 
\begin{equation}\label{u1<=u2-exterior-domain}
u_1(x,t)\le u_2(x,t)\qquad \forall |x|>R_1,\,\, 0<t\leq  T_1.
\end{equation}
For any $R>R_1$, let $\eta$ and $\eta_R$ be as in the proof of Proposition \ref{prop-strong-contraction}. By \eqref{eq-comp-bdry} and the Kato inequality (\cite{K,DK}),  for any $R>R_1$ and $0<t<T_1$, we have
\begin{align}\label{eq-Kato-comp}
\frac{d}{dt} \int_{|x|\geq R_1} \left(u_1-u_2\right)_+ \eta_R dx
&\leq  \int_{|x|\geq R_1} \left(u_1^m-u_2^m\right)_+ \La \eta_R  dx\notag\\
& \leq   {C}{R^{-2}}  \int_{R\leq |x| \leq 2R} \left(u_1^m-u_2^m\right)_+ dx\notag \\
&\leq {C}{R^{-2}}  \int_{R\leq |x| \leq 2R} a(x,t)\left(u_1-u_2\right)_+  dx,
\end{align} 
where
\begin{equation}\label{eq-def-a-app}
a(x,t)= \int_{0}^1\frac{ m\, ds }{(s u_1(x,t) +(1-s) u_2(x,t))^{1-m}}\qquad\forall (x,t) \in (\R^n\setminus\overline{B_{R_1}})\times (0,T_1].
\end{equation}
Note that by \eqref{eq-est-lower-bd}, 
\begin{equation}\label{eq-est-a}
0<\,a(x,t)\,\leq\,  C_0 |x|^2 \qquad\forall (x,t) \in (\R^n\setminus\overline{B_{R_1}})\times (0,T_1]
\end{equation} 
for some constant $C_0>0$. Hence by \eqref{eq-Kato-comp} and \eqref{eq-est-a}, it follows that 
$$
\frac{d}{dt} \int_{|x|\geq R_1}(u_1-u_2)_+\eta_R\,dx
\leq C\int_{R \leq |x|\leq 2R}(u_1-u_2)_+\,dx\quad\forall R>R_1,0<t<T_1.
$$
This together \eqref{eq-comp-bdry} implies that 
\begin{equation}\label{u1-u2-l1-ineqn2}
\int_{|x|\geq R_1}(u_1-u_2)_+ \eta_R (x,t)\,dx
\le C\int_0^t  \int_{R \leq |x|\leq 2R}(u_1-u_2)_+(x,s)\,dx\,ds\quad\forall R>R_1,0<t<T_1.
\end{equation}
Letting  $R\to\infty$ in \eqref{u1-u2-l1-ineqn2}, by   \eqref{u1-u2-l1-norm-bd} and the dominated convergence theorem,  we deduce 
\begin{equation*}
 \int_{|x|\geq R_1} \left(u_1-u_2\right)_+ (x,t ) dx=0 \qquad \forall 0<t< T_1
\end{equation*}
and \eqref{u1<=u2-exterior-domain} follows. 

We now let $w=u_2-u_1$ and observe that the nonnegative function $w$  satisfies 
\begin{align}\label{w-eqn10}
&w_t= \La \big(a\,  w\big) \qquad \qquad \qquad  \qquad \qquad \hbox{ in\,\, $(\R^n\setminus\overline{B_{R_1}})\times (0,T_1]$}\notag\\
\Rightarrow\quad&w_t-a\La w-2\D a\cdot\D w +|\La a|w\ge 0 \quad \hbox{in\,\, $(\R^n\setminus\overline{B_{R_1}})\times (0,T_1]$},
\end{align}
where $a=a(x,t)$ is given by \eqref{eq-def-a-app}. Since for any $R>R_1$,  the   function $a(x,t)$  are continuous on $( \overline{B_{R}}\setminus  {B_{R_1}} )\times [0,T_1]$,  there exists a constant $C_R>0$ such that
\begin{equation}\label{a-lower-bd}
a(x,t)\ge C_R\qquad\forall x\in \overline{B_{R}}\setminus  {B_{R_1}},\,\,0\le t\le T_1.
\end{equation}
Hence by \eqref{eq-est-a} and \eqref{a-lower-bd},  the equation \eqref{w-eqn10}
 for $w$ is uniformly parabolic in $( B_{R}\setminus \overline{B_{R_1}})\times(0, T_1]$ for any $R>R_1$. Therefore by \eqref{eq-comp-bdry} and   the strong maximum principle, we get  that 
\begin{equation*}
w=u_2-u_1>0\qquad\mbox{ in  }(B_R\setminus \overline{B_{R_1}}) \times  (0, T_1].
\end{equation*}
Since $R>R_1$ is arbitrary,  \eqref{u1<u2-exterior-domain} follows.
\end{proof}
 
The following lemma is concerned with  the comparison principle in   bounded domains, and can be proved   by an argument similar to the proof of Lemma 2.3 of \cite{DaK}   and the proof of Lemma \ref{lem-strict-comp-prin}. Hence  we omit its proof. 

\begin{lemma}[Strong comparison principle in bounded domains]\label{lem-strict-comp-prin2}
Let $n\ge 3$, $0<m<\frac{n-2}{n}$, $ R_1>0$ and $T_1>0$.  Let  $u_1$,   $u_2\in  C^{2,1}\big( B_{R_1}\times (0,T_1]\big)\cap C\big(\overline B_{R_1}\times [0,T_1]\big)$  be      solutions  to 
$$u_t=\La u^m \qq\hbox{in\,\, $ {B_{R_1}}\times (0,T_1]$  } $$
  such that   $u_1$,   $u_2$ are positive in $\overline{B_{R_1}}\times [0,T_1]$, and 
\begin{equation*}
\left\{\begin{aligned}
 u_1(x,0)\le u_2(x,0)\qquad&\forall |x|\le R_1;\\
 u_1(x,t)<u_2(x,t)\qquad&\forall |x|=R_1,\,\, 0<  t\leq  T_1.
\end{aligned}\right.
\end{equation*}
Then we have 
\begin{equation*}
u_1(x,t)<u_2(x,t)\qq \forall |x|<R_1,\,\, 0<  t\leq  T_1.
\end{equation*}
\end{lemma}


\end{document}